\documentclass[11pt,reqno]{amsart}
 \usepackage{amsgen, amstext,amsbsy,amsopn, amsthm, amsfonts,amssymb,amscd,amsmath,euscript,enumerate,url,verbatim,calc,xypic,tikz}
\usepackage{hyperref}
\usepackage{MnSymbol}
\usepackage{pgfkeys}
\usepackage{mathtools}
\usepackage{eqnarray,amsmath}

\def\multiset#1#2{\ensuremath{\left(\kern-.3em\left(\genfrac{}{}{0pt}{}{#1}{#2}\right)\kern-.3em\right)}}

\usetikzlibrary{arrows}
\newcommand{\midarrow}{\tikz \draw[-triangle 90] (0,0) -- +(.1,0);}
 \usepackage{latexsym}
 \usepackage{graphics}
 \usepackage{color}
\usepackage{lastpage}
\usepackage{fancyhdr}
\usepackage{multirow}
\allowdisplaybreaks
\usepackage{graphicx}
\graphicspath{ {F:/IMAGES/} }

  \newcommand{\Ass}{\operatorname{Ass}}

 \newcommand{\supp}{\operatorname{supp}}
 
\newcommand{\rad}{\operatorname{rad}}

   \newcommand{\lcm}{\operatorname{lcm}}

\newcommand{\proset}{\,\mathrel{\lower 4pt\hbox{$\scriptscriptstyle/$}
\mkern -14mu\subseteq }\,} 

 \newtheorem{theorem}{Theorem}[section]
 \newtheorem{corollary}[theorem]{Corollary}
 \newtheorem{lemma}[theorem]{Lemma}

\usepackage{amsmath}
\newtheorem{notation}[theorem]{Notation}

 \theoremstyle{definition}
 
 \newtheorem{remark}[theorem]{Remark}
 \newtheorem{definition}[theorem]{Definition}

 \newtheorem{example}[theorem]{Example}
\title[Properties of symbolic powers of edge ideals of weighted oriented graphs] {Properties of symbolic powers of edge ideals of weighted oriented graphs}
\author[M. Mandal and D.K. Pradhan ]{Mousumi Mandal$^*$ and Dipak Kumar Pradhan }

 \thanks{$^*$ Supported by SERB(DST) grant No.: $\mbox{MTR}/2020/000429$, India}
\thanks{AMS Classification 2010: 05C22, 05C25, 05C38, 05E40}
\address{Department of Mathematics, Indian Institute of Technology Kharagpur, 721302, India} \email{mousumi@maths.iitkgp.ac.in}
\address{Department of Mathematics, Indian Institute of Technology Kharagpur, 721302, India}\email{dipakkumar@iitkgp.ac.in}    

\begin{document}
\maketitle

\begin{abstract}
Let $D$ be a weighted oriented graph and $I(D)$ be its edge ideal.
 We provide one method to find all the minimal generators of $ I_{\subseteq C} $,  where   $ C $ is a  maximal strong vertex cover of  $D$ and  $ I_{\subseteq C} $ is the intersections of irreducible ideals
associated to the strong vertex covers contained in  $C$.
If $ D^{\prime} $ is an induced digraph of $D$, under a certain condition on the strong vertex covers of $ D^{\prime} $ and $D$, we show that  $ {I(D^{\prime})}^{(s)} \neq {I(D^{\prime})}^s $ for some $s \geq 2$ implies $ {I(D)}^{(s)} \neq {I(D)}^s $.  We provide the necessary and sufficient condition for the equality of ordinary and symbolic powers of edge ideal of the union of two naturally oriented paths  with a common sink vertex.  We characterize  all the maximal strong vertex covers of $D$  such that at most one edge is oriented into each of its vertices  and $w(x) \geq 2$ if $\deg_D(x)\geq 2 $ for all $x \in V(D)$.
Finally, if $ D $ is a weighted  rooted tree with the
degree of root is $ 1 $ and $ w(x) \geq 2 $ when $ \deg_D(x) \geq  2 $ for all $ x \in V(D) $, we show that  $ {I(D)}^{(s)} = {I(D)}^s $ for all $s \geq 2$.

\end{abstract}

\section{Introduction}

 Let $k$ be a field and $R=k[x_1,\ldots ,x_n]$ be a polynomial ring in $n$ variables. Let $I$ be a homogeneous ideal of $R$. Then for $s\geq 1$, the $s$-th symbolic power of $I$ is defined as $I^{(s)}=\displaystyle{\bigcap_{P\in \Ass I}(I^sR_P\cap R)}.$ We refer  \cite{huneke} to the reader to survey some known results of symbolic powers of ideals. 
\vspace*{0.1cm}\\ 
 A weighted oriented graph  $ D $ with underlying graph  $ G $, is a triplet $ (V (D), E(D), w)$ whose
vertex set is $V(D) = V(G) $, edge set is $E(D)= \{ (x_i, x_j)~|~\{x_i, x_j \} \in E(G)  \}$ and the weight function $ w : V (D) \longrightarrow \mathbb N$. If $ e= (x,y) \in E(D) $, then $ x $ is the initial vertex and $ y $ is the terminal vertex of the directed edge    $ e $.      
The weight of a vertex $x_i\in V(D)$ is $w(x_i)$  denoted by $w_i$ or $ w_{x_i}.$ If a vertex $x_i$ of $D$ is a source (i.e., has only arrows leaving $x_i$), we fix $w_i = 1$. The edge ideal of $D$ is denoted by $I(D)$ and is defined as 
$I(D)=(x_ix_j^{w_j}|(x_i,x_j)\in E(D)).$  This ideal was first studied in \cite{gimenez,pitones}.   Recently in  \cite{grisalde}, the authors give the classification of some normally torsion-free edge ideals of weighted oriented  graphs, where the $s-$th symbolic power of $ I $ defined using the associated
minimal primes of $ I $.

The difficulty in the study of symbolic powers of edge ideals of weighted oriented graphs depends upon the structures of irreducible ideals associated to the strong vertex covers and the number of strong vertex covers. In  general, the number of  strong vertex covers of a weighted oriented graph  is greater  than   the number of minimal vertex covers of its underlying graph and it occurs due to the weights on its vertices and orientation of its edges. So the study of symbolic powers of edge ideals of weighted oriented graphs is harder than simple graphs. In this paper, we provide some methods to study the symbolic powers of their edge ideals.  
In \cite{mandal1}, we see that, if the set of all vertices of a weighted oriented graph  forms a  strong vertex cover,    all the ordinary and symbolic powers
of its edge ideal  coincide. Comparision of ordinary and symbolic powers has been done for several classes of weighted oriented graphs in  \cite{kanoy},  \cite{kanoy1}  and \cite{mandal2}. In all those papers, to compute the symbolic powers, the authors always find  the minimal generators of the intersections of irreducible ideals
associated to the strong vertex covers contained in a maximal strong vertex cover.  
 In this paper,  we  give a direct formula for that in  Theorem \ref{minimalgenerator}  and it works for any weighted oriented graph.   We  compare the ordinary and  symbolic powers of edge ideals of weighted oriented graphs  by studying   the ordinary and  symbolic powers of edge ideals of their  induced digraphs. In \cite{minh}, if $ H $ is an induced subgraph of $G$,  it is known that   $ {I(H)}^{(s)} \neq {I(H)}^s $ for some $s \geq 2$ implies $ {I(G)}^{(s)} \neq {I(G)}^s $. We extend this result to weighted oriented graphs. If $ D^{\prime} $ is an induced digraph of $D$, under a certain condition on the strong vertex covers of $ D^{\prime} $ and $D$, we show that  $ {I(D^{\prime})}^{(s)} \neq {I(D^{\prime})}^s $ for some $s \geq 2$ implies $ {I(D)}^{(s)} \neq {I(D)}^s $ (see  Theorem \ref{induced}). We apply this result to
 compare the ordinary and symbolic powers of edge ideals of weighted oriented paths (see Theorem \ref{inducedpath}). In Theorem \ref{path6r}, we give the necessary and sufficient condition for the equality of ordinary and symbolic powers of edge ideal of union of two naturally oriented paths  with a common sink.

  The main problem in the computation of symbolic power is to find all the maximal strong vertex covers. In  \cite[Lemma 47]{pitones}, 
  Pitones et al. proved  that  $ \{{x_{i_1}},\ldots,{x_{i_s}}\} $ is a vertex cover of $ D $ 	if $  I(D) \subseteq (x_{i_1}^{a_1},\ldots,x_{i_s}^{a_s})$.  We  identify that, if  $a_j = w({x_{i_j}})$ and  $s$ is the least possible value, then $ \{{x_{i_1}},\ldots,{x_{i_s}}\} $ is a maximal strong vertex cover of $ D $  (see Lemma \ref{maximal}). In Theorem \ref{maximal2}, we  prove the converse of Lemma \ref{maximal} is also true under the assumption  `` at most one edge is oriented into each   vertex of $D$  and $w(x) \geq 2$ if $\deg_D(x)\geq 2 $ for all $x \in V(D)$".    Recently in \cite{kanoy}, Banerjee et al. prove the equality of ordinary and symbolic powers of a certain class of weighted rooted trees. 
  In  Theorem \ref{tree},
 we extend this result to prove that ``
   if $ D $ is a weighted  rooted tree with
 degree of root is $ 1 $ and $ w(x) \geq 2 $ whenever $ \deg_D(x) \geq  2 $ for all $ x \in V(D) $, then  $ {I(D)}^{(s)} = {I(D)}^s $ for all $s \geq 2$".
 


\section{Preliminaries}

In this section, we provide some definitions and results for the weighted oriented graphs. By the result of Ha, Nguyen, Trung, and Trung in \cite{H.T-2016}, we may
assume that the underlying graph $ G $ of the weighted oriented graph $ D $ is
connected.

\begin{definition}
	A vertex cover $ C $ of $ D $ is a subset of $  V(D)  $ such that if $ (x, y) \in E(D) ,$ then
	$ x \in C $ or  $ y \in C . $ A vertex cover $ C $ of $ D $ is minimal if each proper subset of $ C $ is not a
	vertex cover of $ D. $ We set $(C)$ to be the ideal generated by the vertices of $C.$
\end{definition}
\begin{definition}
	Let $ x $ be a vertex of a weighted oriented graph $ D, $ then the sets $ N_D^+ (x) =	\{y ~|~ (x, y) \in E(D)\}  $ and  $ N_D^- (x) = \{y ~|~ (y, x) \in E(D)\}  $ are called the out-neighbourhood and the in-neighbourhood of $ x ,$ respectively. Moreover, the neighbourhood of $ x $ is the set $ N_D(x) = N_D^+ (x)\cup N_D^- (x) $ and denote $ N_D[x] = N_D (x)\cup \{x\} .$ Define $\deg_D(x) = |N_D(x)|$ for  $ x \in V(D) $ and $\deg_D(x)$ is known as the degree of the vertex $x$ in $D$. 
	A vertex $ x \in V(D) $ is called a source vertex if $N_D (x)= N_D^+ (x) .$ A vertex $ x \in V(D) $ is called a sink vertex if $N_D (x)= N_D^- (x) .$ We assume
	that the weights of source vertices are trivial.
	We set $V^+(D)$ as the set of vertices of $D$ with non-trivial weights.

\end{definition}
\begin{definition}	
	
	For $T \subset V(D) ,$ we
	define the  induced  digraph $D_T = (V( D_T), E(D_T), w)$ of $D$ on $T$ to be the
	weighted oriented graph such that $V (D_T) = T$ and for any
	$ u, v \in V (D_T),  $  $ (u,v) \in E(D_T)$  if and only if
	$(u,v) \in E(D)$.
	Here $ D_T = (V (D_T), E(D_T), w_T) $
	is a weighted oriented graph with the same orientation  as in $D$ and for any $ u  \in V (D_T), $
	if $ u $ is not a source in $ D_T, $ then its weight equals to the weight of $ u $ in $D,$ otherwise, its
	weight in $ D_T $ is $ 1. $ For a subset $W \subset V(D) $, we  define $ D \setminus  W $ to
	be the  induced digraph of $ D $ on $V(D) \setminus W$.

\end{definition}
\begin{definition}\cite[Definition 4]{pitones}
Let $ C $ be a vertex cover of a weighted oriented graph $ D$. We define the following three sets that form a partition of $ C $\vspace*{0.2cm}\\
\hspace*{3cm}$ L_1^D(C) = \{x \in C ~|~ N_D^+ (x) \cap C^c \neq        \phi \}, $ \vspace*{0.2cm}\\
\hspace*{2.85cm} $L_2^D(C) = \{x \in C ~|~x\notin L_1^D(C) ~\mbox{and}~  N_D^-(x) \cap C^c \neq   \phi \}$ and  \vspace*{0.2cm}\\ 
\hspace*{2.8cm}  $ L_3^D(C) = C \setminus (L_1^D(C) \cup L_2^D(C))$, \vspace*{0.2cm}\\
 where $ C^c $ is the complement of $  C ,$ i.e., $ C^c = V(D) \setminus C. $
\end{definition}
%
 
 \begin{lemma}\cite[Proposition 5]{pitones}\label{L3}  
 If   $ C $ is a vertex cover of $ D, $ then $ L_3^D(C) =\{x\in  C~|~N_D(x) \subset C\}.	 $  
 \end{lemma}

\begin{definition}\cite[Definition 7]{pitones}
A vertex cover $ C $ of $ D $ is strong if for each $ x \in L_3^D(C)  $ there is $ (y, x) \in 
E(D) $ such that $ y \in L_2^D(C) \cup L_3^D(C)$ with $  y \in V^+(D)$   (i.e., $ w(y) \neq  1 $).   
\end{definition}

\begin{remark}\cite[Remark 8, Proposition 5]{pitones}\label{s.v.1}
	A vertex cover	$ C $ of $D$  is strong if and only if for each $ x \in L_3^D(C),$   we have $N_D^{-}(x) \cap V^+(D) \cap [C\setminus {L_1^D{(C)}} ] \neq \phi.$  
\end{remark}


\begin{definition}
A strong vertex cover $C$ of $D$ is said to be a maximal strong vertex cover of $D$ if it is not contained in any  other strong vertex cover of $D.$
\end{definition}



%
%

\begin{definition}\cite[ Definition 19]{pitones}\label{definition19}   
Let $ C $ be a vertex cover of $ D. $  The irreducible ideal associated to $ C $ is the ideal 
$ I_C :=(L_1^D(C)\cup \{x_j^{w(x_j)} |x_j \in  L_2^D(C) \cup  L_3^D(C)    \}).$	
\end{definition}

\begin{lemma}\cite[Lemma 20]{pitones}\label{edge}
	Let $ D $ be a  weighted oriented graph. Then $I(D) \subseteq I_C$, for each vertex cover $ C $ of $ D $.  	
\end{lemma}

The next lemma gives us the irreducible decomposition of the edge ideal of a  weighted oriented graph $ D $ in terms of irreducible ideals associated with the strong vertex covers of $D$.

\begin{lemma}\cite[Theorem 25, Remark 26]{pitones}\label{s.v.2}    
Let $ D $ be  a weighted oriented graph and $C_1,\ldots, C_s$ are the strong vertex covers of $ D ,$ then  the irredundant irreducible decomposition of $ I(D) $ is
    $$I(D) =  I_{C_1} \cap\cdots\cap I_{C_s} $$ 
where each $ I_{C_i} = ( L_1^D(C_i) \cup \{x_j^{w(x_j)}~|~x_j \in L_2^D(C_i) \cup L_3^D(C_i)\} ) ,$   $ \rad(I_{C_i})=P_i = (C_i)$.
\end{lemma}
\begin{corollary}\cite[Remark 26]{pitones}\label{s.v.3}
	Let $ D $ be a weighted oriented graph. Then $ P $ is an associated 
	prime of $ I(D) $ if and only if $ P = (C) $ for some strong vertex cover $ C $ of $ D. $
\end{corollary}

%


%
%
%

Let $ I \subset R$ and $ I = Q_1\cap \cdots \cap Q_m $ be a primary decomposition of ideal $I$. For $ P \in \Ass(R/I), $
we denote $ Q_{\subseteq P} $ to be the intersection of all $ Q_i $ with 
$ \sqrt{Q_i} \subseteq P.  $ If $C$ is a  strong vertex cover of a weighted oriented graph $ D $, then $(C)  \in \Ass(R/I(D))$. 
We denote   $ I_{\subseteq {C}} $ as  $ I_{\subseteq {(C)}} $. In the following lemma, we express the \cite[Theorem 3.7]{cooper} for  edge ideals of weighted oriented graphs.
\begin{lemma}\cite[Theorem 3.7]{cooper}\label{cooper} Let $I$ be the edge ideal of a weighted oriented graph $ D $ and $C_1,\ldots,C_r$ be the maximal strong vertex covers of $D.$ Then  $$I^{(s)}=(I_{\subseteq {C_1}})^s \cap\cdots\cap (I_{\subseteq {C_r}})^s.$$
\end{lemma}

The following three lemmas are useful to get the necessary and sufficient condition for the equality of ordinary and symbolic powers of edge ideals of weighted oriented paths.

\begin{definition}
A path is said to be naturally oriented, if all of its edges are oriented in one direction.

\end{definition}
 
\begin{lemma}\cite[Lemma 3.8]{mandal2}\label{atmost}  
	Let $ D $ be a  weighted oriented graph such that at most one edge is oriented into each vertex.
	Let	$D^{\prime}$ be an induced weighted naturally oriented path  of length $3$ of $D$ with  $V(D^{\prime}) = \{x_{i-1},x_i,x_{i+1},x_{i+2}\}$, $E(D^{\prime}) = \{ (x_{j},x_{j+1}) ~|~ i-1 \leq j \leq  i+1 \}$, $w(x_i) \geq 2$ and $w(x_{i+1}) = 1$.
	Then $ {I(D)}^{(3)} \neq {I(D)}^3 $.           	
\end{lemma}

\begin{theorem}\cite[Theorem 3.6]{kanoy}\label{kanoy2}
	Let $ D $ be a weighted naturally oriented path with
	$ V(D) = \{x_1, x_2, x_3, \ldots , x_n\}$  and $E(D) = \{(x_i,x_{i+1}) ~|~ 1 \leq i \leq n-1 \}$.
	Then $ I(D)^{(s)} = I(D)^{s} $ for all $s \geq 2$ 
	if and only if it satisfies the condition ``if $w(x_j) \geq 2$ for some $1 < j < n$ then  $w(x_i) \geq 2$ for some $j \leq i \leq n-1$".  
	
\end{theorem}

\begin{lemma}\cite[Corollary 4.6]{mandal2}\label{cor3} 
Let $ D $ be a  weighted oriented graph. Let $ D^{\prime} $ be the  weighted oriented graph obtained from  $D$ after replacing  the non-trivial weights of sink vertices by
trivial weights.	Then     $ {{I(D)}^{(s)}} = {{I(D)}^{s}} $  if and only if $ {{I(D^{\prime})}^{(s)}} = {{I(D^{\prime})}^{s}} $  for each $s \geq 1$.    	
\end{lemma}

%


\begin{notation}    
	Let $I \subset  k[x_1,\ldots,x_n]$ be a monomial ideal. We set $ \mathcal{G}(I) $ be the set of minimal generators of the ideal $I.$ If $J$ is a set of some elements of $I$, then $\langle J \rangle$  denoted as the  ideal generated by the elements of $J$.		
\end{notation}

%
%

\begin{notation}
	Let  $g \in  k[x_1,\ldots,x_n]$ be a monomial.  We define support of $g$ $=  \{x_i:x_i \divides g\}  $ and we denote it by $\supp(g).$

\end{notation}

\section{Symbolic powers of weighted oriented graphs}

  In this section,  we give one method to find all the minimal generators of $ I_{\subseteq C} $,  where   $ C $ is any  maximal strong vertex cover of a weighted oriented graph $D$.

Let	$ C $ be a vertex cover of $D$. We call that a vertex  $ x \in {L_3^D{(C)}}$ satisfies the SVC condition on $C$ if and only if $N_D^{-}(x) \cap V^+(D) \cap [{L_2^D{(C)}} \cup   {L_3^D{(C)}} ] \neq \phi.$  	By Remark \ref{s.v.1},  $ C $   is strong if and only if  each $ x \in L_3^D(C)$  satisfies the SVC condition on $C$. 

Next we give some  lemmas, which are  useful to prove our main results.

\begin{lemma}\label{svc1}
	Let $ D $ be  a  weighted  oriented  graph. Let $C_1$  and  $C_2$ be two vertex covers of $D  $ such that $C_1  \subset  C_2$. If $x \in L_3^D{(C_1)}$ satisfies the SVC condition on $C_1$, then  $x \in L_3^D{(C_2)}$ and it satisfies  the    SVC condition on $C_2$.        	
\end{lemma}

\begin{proof}
Since  $x \in L_3^D{(C_1)}$, by Lemma \ref{L3}, we have $x \in L_3^D{(C_2)}$. Given that $ x \in {L_3^D{(C_1)}}$ satisfies the SVC condition on $C_1$. That means there exists  $y \in  N_D^{-}(x) \cap V^+(D) \cap [{L_2^D{(C_1)}} \cup   {L_3^D{(C_1)}} ] .$ Here $y \notin  {L_1^D{(C_1)}}$ implies $N_D^{+}(y) \cap  C_1^c = \phi $. Since   $C_1  \subset  C_2$, we have  $N_D^{+}(y) \cap  C_2^c = \phi $ and so   $y \notin  {L_1^D{(C_2)}}$. Then  $y \in  N_D^{-}(x) \cap V^+(D) \cap [{L_2^D{(C_2)}} \cup   {L_3^D{(C_2)}} ] .$  Hence  $x$ satisfies the SVC condition on $C_2$.   	
\end{proof}

\begin{lemma}\label{svc2}
	Let $ D $ be  a  weighted  oriented  graph. Let $C$ be a  vertex cover of $D.$ Then there exists a strong vertex cover $C^{\prime} \subseteq C$  of $D$ such that there is no strong vertex cover $C^{\prime\prime}  \subseteq C$ of $D$ with $C^{\prime}  \subsetneq C^{\prime\prime}$.     	
\end{lemma}

\begin{proof}
Let $C$ be a  vertex cover of $D$. Let $J $ $\subseteq {L_3^D{(C)}}$ be the set of vertices which does not satisfy the SVC condition on $C$.
If $J= \phi $, then we take  $C^{\prime} = C$. Now we assume $J  \neq \phi $. Let $J = \{x_{l_1},\ldots,x_{l_r}\}$. Choose one element $x_{j_1} \in J  \cap {L_3^D{(C)}}$ and set $C_1= C \setminus \{x_{j_1}\}$. By Remark \ref{L3}, $N_D[x_{j_1}] \subseteq C$ and  so $C_1$ is a vertex cover of $D$. Now, we suppose that there are vertex
covers $C_0,\ldots,C_k$ such that $C_i= C_{i-1} \setminus \{x_{j_i}\}$ and $x_{j_i} \in J  \cap {L_3^D{(C_{i-1})}}   $ for $i=1,\ldots,k$, where $C_0 = C$ and we give the following recursively process: If $ J  \cap {L_3^D{(C_{i-1})}} = \phi,$ then we take $C^{\prime} = C_{i-1}$. Since $|V(D)|$ is finite, the process is finite. Hence there exists $m  \leq   r$ such that $ J  \cap {L_3^D{(C_{m})}} = \phi.$ Then we take $C^{\prime} = C_{m} = C \setminus \{x_{j_1},\ldots, x_{j_m} \} $, where $ \{x_{j_1},\ldots, x_{j_m} \}  \subseteq J$.
\vspace*{0.2cm}\\
Now we claim that $C^{\prime}$ is strong. Let $x_{p} \in {L_3^D{(C^{\prime})}}  $.
 By Lemma \ref{L3},  $x_{p} \in {L_3^D{(C)}}  $   because   $C^{\prime}  \subseteq C$.  Since  $ J  \cap {L_3^D{(C^{\prime})}} = \phi$, we have $x_{p} \notin J  $. 
Then $x_{p} $  satisfies the SVC condition on $C$. Thus there is some  $y_{p} \in N_D^{-}(x_{p}) \cap V^+(D) \cap [{L_2^D{(C)}} \cup   {L_3^D{(C)}} ].$  Note that  $y_{p} \notin  L_1^D{(C)}.$   Suppose $y_{p} \in  L_1^D{(C^{\prime})}.$ That means  $x_{j_i} \in  N_D^{+}(y_{p})  \cap {C^{\prime}}^c   $ for some $i  \in [m]$. Now  $y_{p} \in N_D^{-}(x_{j_i}) \cap V^+(D) \cap [{L_2^D{(C)}} \cup   {L_3^D{(C)}} ]$ and so $x_{j_i} \in {L_3^D{(C)}}  $  satisfies the SVC condition on $C$, which is a contradiction. Therefore $y_{p} \notin  L_1^D{(C^{\prime})}$ and  $y_{p} \in N_D^{-}(x_{p}) \cap V^+(D) \cap [{L_2^D{(C^{\prime})}} \cup   {L_3^D{(C^{\prime})}} ]$, i.e., 
$x_{p} \in {L_3^D{(C^{\prime})}}  $  satisfies the SVC condition on $C^{\prime}$. Hence $C^{\prime}$ is strong.
\vspace*{0.1cm}\\
Suppose  there exists a strong vertex cover $C^{\prime\prime}  \subseteq C$  such that $C^{\prime}  \subsetneq C^{\prime\prime}$. This implies  
  $x_{j_i}  \in     C^{\prime\prime}$ for some $i \in [m]$.  Since $x_{j_i} \notin C^{\prime}  $, we have $N_D(x_{j_i}) \subset C^{\prime}  $.
By Lemma \ref{L3},  $x_{j_i} \in {L_3^D{(C^{\prime\prime})}}  $  and  it  satisfies the SVC condition on $C^{\prime\prime}$ because $C^{\prime\prime}$ is strong.  By Lemma \ref{svc1},  $x_{j_i} \in {L_3^D{(C)}}  $ and it   satisfies the SVC condition on $C$, which is a contradiction.  Hence the proof follows.  
\end{proof}    

\begin{lemma}\label{svc5}
	Let $ D $ be  a  weighted  oriented  graph. Let $C$ be a  vertex cover of $D$ . Then there exists a strong vertex cover $C^{\prime} \subseteq C$  of $D$, where  $x \notin  L_1^{D}(C)$ with $w(x)  \neq 1$  implies $x \notin   L_1^{D}(C^{\prime}) $.           	
\end{lemma}

\begin{proof}
By Lemma \ref{svc2}, there exists a strong
vertex cover $C^{\prime} \subset C$ of D whose each element of $  C \setminus  C^{\prime}$ does not satisfy the
SVC condition on $ C $.	We may assume that $C^{\prime} = C \setminus \{x_{1},\ldots, x_{m} \} $ of $D$.  Let  $x \notin  L_1^{D}(C)$ with $w(x)  \neq 1$.     Suppose $x \in  L_1^{D}(C^{\prime}) $. That means  $x_{j} \in  N_D^{+}(x)  \cap {C^{\prime}}^c   $  for some $j  \in [m]$. Then  $x \in N_D^{-}(x_{j}) \cap V^+(D) \cap [{L_2^D{(C)}} \cup   {L_3^D{(C)}} ]$ and so $x_{j} \in {L_3^D{(C)}}  $  satisfies the SVC condition on $C$, which is a contradiction. Hence $x \notin   L_1^{D}(C^{\prime}) $.
\end{proof}

\begin{corollary}\label{svc3}
	Let $ D $ be  a  weighted  oriented  graph. Let $C$ be a  vertex cover of $D$ with $x_i \notin  L_3^{D}(C) $. Then there exists a strong vertex cover $C^{\prime} \subseteq C$  of $D$ such that $x_i \in C^{\prime}$.
\end{corollary}

\begin{proof}
	By Lemma \ref{svc2},  there exists a strong vertex cover $C^{\prime} \subseteq C$  of $D$.
	By Lemma \ref{L3}, $N_D(x_i)  \nsubseteq  C$  and so $N_D(x_i)  \nsubseteq  C^{\prime}$.  Since  $C^{\prime}  $  is a  vertex cover  of $D$,  we have  $x_i \in  C^{\prime}$.    
\end{proof}

We are now ready for the
main result of this section which describes the minimal generators of $ I_{\subseteq C} $ for a
maximal strong vertex cover $ C $.

 \begin{theorem}\label{minimalgenerator}
	Let $ D $ be a  weighted oriented graph on the vertex set $\{x_1,\ldots,x_n\}$. Let $I = I(D)$ and  $w_i = w(x_i)$ for all $x_i \in V(D)$.  Let $ C $ be a  maximal strong vertex cover of $D$. Then
	$ I_{\subseteq C}  = (L_1^D(C))  + (x_i^{w_i} ~|~ x_i \in  L_2^D(C)     ) + (x_ix_j^{w_j} ~|~ (x_i,x_j) \in E(D),~ x_i \in 
	L_2^{D}(C) \cup 	L_3^{D}(C) ~\mbox{and}~ x_j \in 
	L_3^{D}(C)  )   $.

\end{theorem}

\begin{proof}  
Suppose $x_i \in L_1^{D}(C)$.  Then for any strong vertex cover $C^{\prime} \subseteq C,$ $ N_D^+ (x_i) \cap {C^{\prime}}^c \neq        \phi $   and so $x_i \in L_1^{D}(C^{\prime}).$ Thus $ L_1^{D}(C)  \subseteq     L_1^D(C^{\prime})  $  and hence $   (L_1^D(C))  \subseteq       I_{\subseteq C}  $.  Note that each element of $ L_1^D(C) $  is a minimal generator of $I_{\subseteq C}$.
\vspace*{0.1cm}\\
Suppose $x_i \in L_2^{D}(C)$.
Then for any strong vertex cover $C^{\prime} \subseteq C,$ $ N_D^- (x_i) \cap {C^{\prime}}^c \neq        \phi $   and so $x_i \in L_1^{D}(C^{\prime})\cup L_2^{D}(C^{\prime})$.  Thus  $L_2^{D}(C) \subseteq L_1^{D}(C^{\prime})\cup L_2^{D}(C^{\prime})$  and hence $(x_i^{w_i} ~|~ x_i \in  L_2^D(C)     ) \subseteq       I_{\subseteq C}   $. Notice that that each element of the set $ \{x_i^{w_i} ~|~ x_i \in  L_2^D(C)        \} $  is a minimal generator of $I_{\subseteq C}$.
\vspace*{0.1cm}\\
Suppose $(x_i,x_j)  \in E(D)$ where $x_j \in L_3^{D}(C)$. By Lemma \ref{L3}, $N_D[x_j] \subseteq C$ and so $C_1  =  C \setminus \{x_j\}$  is a vertex cover of $D$.  By Lemma \ref{L3}, $x_i  \notin  L_3^D(C_1)$    and hence by  Corollary \ref{svc3}, there exists a strong vertex cover $C^{\prime} \subseteq C_1$ 
such that $ x_i \in C^{\prime}.$ Here  $x_{j} \in  N_D^{+}(x_i)  \cap {C^{\prime}}^c   $. Thus $x_i \in L_1^{D}(C^{\prime})$ and so $x_i \in \mathcal{G}(I_{C^{\prime} })$. Note that $x_i \notin  I_{C} $,  $x_j^{w_j} \in \mathcal{G}(I_{C})$  and  $x_j^{w_j}  \notin     I_{C^{\prime} }$.
\vspace*{0.1cm}\\
Then $x_ix_j^{w_j} \in \mathcal{G}(I_{C} \cap I_{C^{\prime} }).$ By Lemma \ref{edge}, for any  strong vertex cover $C^{\prime\prime} \subseteq C$,   $x_ix_j^{w_j} \in I_{C^{\prime\prime} }$ and hence  $x_ix_j^{w_j} \in \mathcal{G}(I_{\subseteq C}).$
If $w_j \neq 1$, we have $x_ix_j \notin I_{\subseteq C}$ and so $ x_ix_j^{w_j} $  is the only minimal generator of $I_{\subseteq C}$, which  involves both $x_i$ and $x_j$.  If $w_j = 1$, $ x_ix_j $  is the only minimal generator of $I_{\subseteq C}$, which  involves both $x_i$ and $x_j$.
\vspace*{0.1cm}\\
Suppose $x_i \in L_3^{D}(C)$. Suppose $(x_i,x_j)  \in E(D)$ where $x_j \in L_3^{D}(C)$.
By the previous argument, $C_1  =  C \setminus \{x_j\}$  and   $C_2  =  C \setminus \{x_i\}$  are  vertex covers of $D$.     By Lemma \ref{L3}, $x_i  \notin  L_3^D(C_1)$. Thus by Corollary \ref{svc3}, there exists a strong vertex cover $C^{\prime} \subseteq C_1$ 
such that $ x_i \in C^{\prime}.$ Here  
$x_{j} \in  N_D^{+}(x_i)  \cap {C^{\prime}}^c   $ and so $x_i \in L_1^{D}(C^{\prime})$. Thus $x_i \in \mathcal{G}(I_{C^{\prime} })$. 
By Lemma \ref{L3}, $  N_D(x_j)  \subset  C$.  Since $(x_i,x_j)  \in E(D)$, we have $ N_D^+ (x_j) \cap  {C_2}^c  =  \phi$ and $x_i \in N_D^- (x_j) \cap  {C_2}^c$. Then   $x_j \in L_2^{D}(C_2)$.  Here $x_j \notin L_1^{D}(C_2)  \cup  L_3^{D}(C_2)  $. By Corollary \ref{svc3}, there exists a strong vertex cover $C^{\prime\prime} \subseteq C_2$  such that $ x_j \in C^{\prime\prime}$. If $w_j = 1$,  then
$x_j \in \mathcal{G}(I_{C^{\prime\prime} })$.    
If $w_j \neq 1$,  by Lemma \ref{svc5}, we get $x_j \notin L_1^{D}(C^{\prime\prime})$ and so  
$x_j^{w_j} \in \mathcal{G}(I_{C^{\prime\prime} })$.  In both cases $x_j^{w_j} \in \mathcal{G}(I_{C^{\prime\prime} })$. Note that $x_i \in \mathcal{G}(I_{C^{\prime} }),x_i \notin  I_{C^{\prime\prime} } $  and  $x_j^{w_j}  \notin     I_{C^{\prime} }$.    Then  $x_ix_j^{w_j} \in \mathcal{G}(I_{C^{\prime} } \cap I_{C^{\prime\prime} }).$
By Lemma \ref{edge}, for any  strong vertex cover $C^{\prime\prime\prime} \subseteq C$,   $x_ix_j^{w_j} \in I_{C^{\prime\prime\prime}}$ and hence  $x_ix_j^{w_j} \in \mathcal{G}(I_{\subseteq C}).$
Note that there is no other   minimal generator of $I_{\subseteq C}$, which  involves both $x_i$ and $x_j$.

Hence  
$  (L_1^D(C))  + (x_i^{w_i} ~|~ x_i \in  L_2^D(C)     ) + (x_ix_j^{w_j} ~|~ (x_i,x_j) \in E(D),~ x_i \in 
L_2^{D}(C) \cup 	L_3^{D}(C) ~\mbox{and}~ x_j \in 
L_3^{D}(C)  )  \subseteq  I_{\subseteq C}   $.
\vspace*{0.1cm}\\
To complete the proof, it is enough to prove the following two statements:
  
$ (1) $ There  is no minimal generator  of  $I_{\subseteq C}$,  which  involves more than two vertices.	

$ (2) $ There  is no minimal generator  of  $I_{\subseteq C}$,  which  involves  two non-adjacent vertices.

$ (1) $  Suppose there  exists a minimal generator $f$ of $I_{\subseteq C}$,  which  involves more than two vertices.  Since $f$ is   minimal, we can assume that no element of $\supp(f)$ $\in$ $L_1^{D}(C)$  and  no element of $\supp(f)$ with trivial weight $\in$ $L_2^{D}(C)$.        
Let $f=x_1^{a_1} \cdots x_r^{a_r} y_1^{b_1} \cdots y_s^{b_s} z_1^{c_1} \cdots z_t^{c_t}$ where $\{x_1,\ldots,x_r\} \subseteq L_2^{D}(C)$,  $\{y_1,\ldots,y_s,z_1,\ldots,z_t \} \subseteq L_3^{D}(C)$, $a_i = 1 $  or $ w_{x_i}  $  for $1 \leq i \leq r$, $b_i = 1$ or $ w_{y_i} $  for $1 \leq i \leq s$, $c_i = 1 $ or $ w_{z_i} $  for $1 \leq i \leq t$  and $r+s+t \geq 3$.
Without loss of generality we can assume that  $b_i = 1$ with $w_{y_i} \neq 1$   for $1 \leq i \leq s$  and  $c_i = w_{z_i}$ for $1 \leq i \leq t$. By our assumption   $w_{x_i} \neq 1$ and    so $x_i^{w_{x_i}} \in \mathcal{G}(I_{\subseteq C})$  for $1 \leq i \leq r$. Since $f$ is minimal,  $a_i = 1$  for $1 \leq i \leq r$. Hence $f=x_1\cdots x_r y_1\cdots y_s z_1^{w_{z_1}} \cdots z_t^{w_{z_t}}$. If $t=0,$ $f \notin I_C$,  which is a contradiction.  Now we assume $t  \neq    0.$
If two  $z_i$'s (say $z_k$ and $z_l$) are adjacent,  $z_kz_l^{w_{z_l}}$ or  $z_lz_k^{w_{z_k}} \in \mathcal{G}(I_{\subseteq C})$ and so  $f$ is not minimal. Therefore no two $z_i$'s are adjacent. By Lemma \ref{L3}, $N_D[z_i] \subset C$  for $1 \leq i \leq t$. Thus $C_1 = C \setminus \{z_1,\ldots,z_t\}$ is a vertex cover of $D$. If  $x_i \in L_1^{D}(C_1)$ for some $i \in [r]$, that means there  exists some  $  z_j \in N_D^+ (x_i) \cap {C_1}^c.$ Since $(x_i,z_j) \in E(D)$,  $x_iz_j^{w_j} \in \mathcal{G}(I_{\subseteq C})$ and it contradicts the fact that, $f$ is minimal. By the similar argument, if  $y_i \in L_1^{D}(C_1)$ for some $i \in [s]$,  we get a contradiction. Therefore each of $x_i$  and   $y_i \notin L_1^{D}(C_1)$.  By Lemma \ref{svc5}, there exists  a strong vertex cover $C^{\prime} \subseteq C_1 $, where each of $x_i$  and $y_i \notin L_1^{D}(C^{\prime})$. So $I_{{C^{\prime}}} = (x_1^{w_{x_1}},\ldots, x_r^{w_{x_r}}, y_1^{w_{y_1}},\ldots,   y_s^{w_{y_s}},\ldots  )$. Note that $f \notin I_{C^{\prime}}$, which is a contradiction. Hence there does not  exist any minimal generator of $I_{\subseteq C}$, which involves more than two vertices.

$(2)$ By the similar argument as in $(1)$,  we can show that  there is  no minimal generator  of $I_{\subseteq C}$, which involves    two  non-adjacent vertices of $D$.      
\end{proof}

%
%
%
%
%
%
%
%
%

Next we see  some  applications of the above theorem.

\begin{definition}
	A rooted tree is
	an oriented tree in which all edges are oriented  away from  the root.
\end{definition}
\begin{example} 
Consider the  weighted rooted  tree $D$ with degree of root is $1$, 	
\begin{figure}[!ht]
		\begin{tikzpicture}[scale=01]
			\begin{scope}[ thick, every node/.style={sloped,allow upside down}] 
				\definecolor{ultramarine}{rgb}{0.07, 0.04, 0.56} 
				\definecolor{zaffre}{rgb}{0.0, 0.08, 0.66}
				
			\draw [fill, black](-1.5,0) --node {\midarrow}(0.5,0);  
			\draw [fill, black](0.5,0) --node {\midarrow}(2,1);
			\draw [fill, black](0.5,0) --node {\midarrow}(2,-1);	
			\draw [fill, black](2,1) --node {\midarrow}(3.3,1.5);
			\draw [fill, black](2,1) --node {\midarrow}(3.3,0.5);
			\draw [fill, black](2,-1) --node {\midarrow}(3.3,-1.5);
			\draw [fill, black](2,-1) --node {\midarrow}(3.3,-0.5);

	        \draw [fill, black](3.3,1.5) --node {\midarrow}(4.6,2);
	        \draw [fill, black](3.3,1.5) --node {\midarrow}(4.6,1.3);   
		    \draw [fill, black](3.3,0.5) --node {\midarrow}(4.6,0.9);
			\draw [fill, black](3.3,0.5) --node {\midarrow}(4.6,0.2);
			
			\draw [fill, black](3.3,-1.5) --node {\midarrow}(4.6,-2);
			\draw [fill, black](3.3,-1.5) --node {\midarrow}(4.6,-1.3);   
			\draw [fill, black](3.3,-0.5) --node {\midarrow}(4.6,-0.9);
			\draw [fill, black](3.3,-0.5) --node {\midarrow}(4.6,-0.2);

	\draw[fill] [fill] (-1.5,0) circle [radius=0.05];
	\draw[fill] [fill] (0.5,0) circle [radius=0.05];
	\draw[fill] [fill] (2,1) circle [radius=0.05];
	\draw[fill] [fill] (2,-1) circle [radius=0.05];
	\draw[fill] [fill] (3.3,1.5) circle [radius=0.05];
	\draw[fill] [fill] (3.3,0.5) circle [radius=0.05];
	\draw[fill] [fill] (3.3,-1.5) circle [radius=0.05];
    \draw[fill] [fill] (3.3,-0.5) circle [radius=0.05];

   \draw[fill] [fill] (4.6,2) circle [radius=0.05];
   \draw[fill] [fill] (4.6,1.3) circle [radius=0.05];
   \draw[fill] [fill] (4.6,0.9) circle [radius=0.05];
   \draw[fill] [fill] (4.6,0.2) circle [radius=0.05];
  \draw[fill] [fill] (4.6,-2) circle [radius=0.05];
  \draw[fill] [fill] (4.6,-1.3) circle [radius=0.05];
  \draw[fill] [fill] (4.6,-0.9) circle [radius=0.05];
  \draw[fill] [fill] (4.6,-0.2) circle [radius=0.05];  
	
%
			    \node at (-1.5,0.5) {$x_0$};
				\node at (0.5,0.5) {$x_1$};
				\node at (2,1.5) {$x_{2}$};  
				\node at (2,-1.5) {$x_{3}$};
				
\node at (-1.5,-0.5) {$w=1$};
\node at (0.5,-0.5) {$w \neq 1$};
\node at (2,0.5) {$w  \neq  1$};  
\node at (2,-0.5) {$w \neq 1$};    				
				
\node at (3.3,1.9) {$w  \neq  1$};  
\node at (3.3,0.9) {$w \neq 1$};				
\node at (3.3,-1.9) {$w  \neq  1$};  
\node at (3.3,-0.9) {$w \neq 1$};

%
%

				\node at (2,-2.8) {$D$};
				

			\end{scope}
		\end{tikzpicture}
		\caption{A weighted  rooted  tree  $D$}\label{fig.1}
	\end{figure}
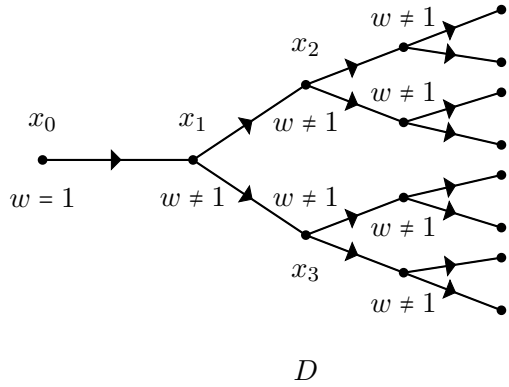
as in  Figure \ref{fig.1}. Let $I=I(D)$. Note that $C = V(D) \setminus  \{x_1\}$ is a   vertex  cover of $D$. Here $ L_1^{D}(C) = \{x_0\}$, $ L_2^{D}(C) = \{x_2,x_3\}$ and  $ L_3^{D}(C) = V(D) \setminus  \{x_0,x_1,x_2,x_3\}$.  By the definition of $D$,  we can see that each element of $ L_3^{D}(C) $ satisfies the SVC condition on $C$  and so $C$ is  strong in $D$.  Notice that $ V(D) $ is a   vertex  cover of $D$ and  $ L_3^{D}(V(D)) = V(D) $.  Since $N_D^-(x_0) = \phi$, $x_0$ does not satisfy the SVC condition on $V(D)$. Therefore $V(D)$ is not    strong in $D$.    Hence  $C$ is  maximal.  Let $D_1 = D \setminus \{x_0,x_1\}$.  Then by  Theorem \ref{minimalgenerator},   we have   $\displaystyle{I_{\subseteq C}   = (x_0,x_2^{w_2},x_3^{w_3})+I(D_1)}$.	
\end{example}

\pagebreak

\begin{remark}
	In a weighted oriented graph, if we know  all the maximal strong vertex covers, then by the Theorem \ref{minimalgenerator} and Lemma \ref{cooper}, we can find the symbolic powers of its edge ideal.

\end{remark}

\section{Symbolic powers of induced weighted oriented graphs}
In this section,  we see that, by studying    the symbolic powers of edge ideal of an  induced digraph of a weighted oriented graph  $D$, we can get information about the symbolic powers  of edge ideal of  $D$.

By \cite[Corollary 2.7]{minh}, if $ H $ be an induced subgraph of a simple graph $G$, then $ {I(H)}^{(s)} \neq {I(H)}^s $ for some $s \geq 2,$ implies $ {I(G)}^{(s)} \neq {I(G)}^s $. In general, this property may not hold for weighted oriented graphs.  But under certain condition on the strong vertex covers,  we extend this result for weighted oriented graphs  in Theorem \ref{induced}.

\begin{remark}
Let  $D^{\prime}$  is an induced  digraph of $D$. If one source vertex of $D^{\prime}$ is not source in  $D$,    then  $ {I(D^{\prime})}^{(s)} \neq {I(D^{\prime})}^s $ for some $s \geq 2,$ may not imply  $ {I(D)}^{(s)} \neq {I(D)}^s $.  For example consider the  weighted oriented  paths $D$ and $D^{\prime}$ as in Figure \ref{fig.2}.   Then $I(D) = (x_1x_2^3,x_2x_3,x_3x_4,x_4^3x_5)$ and $I(D^{\prime}) = (x_1x_2^3,x_2x_3,x_3x_4)$. Note that $D^{\prime}$  is an induced  path of $D$.  Here $w(x_4) = 1$ in $D^{\prime}$  but  $w(x_4)=3$ in $D$. 
Using Macaulay 2, we see  that $ {I(D^{\prime})}^{(2)} \neq {I(D^{\prime})}^2 $, but   $ {I(D)}^{(2)} = {I(D)}^2 $.

\begin{figure}[!ht]
		\begin{tikzpicture}[scale=0.7]
			\begin{scope}[ thick, every node/.style={sloped,allow upside down}] 
				\definecolor{ultramarine}{rgb}{0.07, 0.04, 0.56} 
				\definecolor{zaffre}{rgb}{0.0, 0.08, 0.66}
				
				\draw [fill, black](0,0) --node {\midarrow}(1.5,0);  
				\draw [fill, black](1.5,0) --node {\midarrow}(3,0);
				\draw [fill, black](4.5,0) --node {\midarrow}(3,0);
				\draw[fill, black]  (6,0) --node {\midarrow}(4.5,0);

				\draw[fill] [fill] (0,0) circle [radius=0.04];
				\draw[fill] [fill] (1.5,0) circle [radius=0.04];
				\draw[fill] [fill] (3,0) circle [radius=0.04];
				\draw[fill] [fill] (4.5,0) circle [radius=0.04];
				\draw[fill] [fill] (6,0) circle [radius=0.04];

				\node at (0,0.5) {$x_1$};
				\node at (1.5,0.5) {$x_2$};
				\node at (3,0.5) {$x_{3}$};  
				\node at (4.5,0.5) {$x_{4}$};
				\node at (6,0.5) {$x_{5}$};    
				
				\node at (0,-0.5) {$1$};
				\node at (1.5,-0.5) {$3$};
				\node at (3,-0.5) {$1$};  
				\node at (4.5,-0.5) {$3$};
				\node at (6,-0.5) {$1$};

				\draw [fill, black](8,0) --node {\midarrow}(9.5,0);  
				\draw [fill, black](9.5,0) --node {\midarrow}(11,0);
				\draw [fill, black](12.5,0) --node {\midarrow}(11,0);

				\draw[fill] [fill] (8,0) circle [radius=0.04];
				\draw[fill] [fill] (9.5,0) circle [radius=0.04];
				\draw[fill] [fill] (11,0) circle [radius=0.04];
				\draw[fill] [fill] (12.5,0) circle [radius=0.04];

				\node at (8,0.5) {$x_1$};
				\node at (9.5,0.5) {$x_2$};
				\node at (11,0.5) {$x_{3}$};  
				\node at (12.5,0.5) {$x_{4}$};
				
				\node at (8,-0.5) {$1$};
				\node at (9.5,-0.5) {$3$};
				\node at (11,-0.5) {$1$};  
				\node at (12.5,-0.5) {$1$};

				\node at (3.1,-1.2) {$D$};
				\node at (10.3,-1.2) {$D^{\prime}$};	

			\end{scope}
		\end{tikzpicture}
		\caption{A weighted oriented path $D$ containing an induced weighted oriented path  $D^{\prime}$.}\label{fig.2}
	\end{figure}
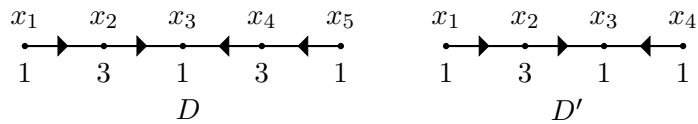

	\hspace*{0.5cm} If all source vertices of $D^{\prime}$ are source in $D$, we may have $ {I(D^{\prime})}^{(s)} \neq {I(D^{\prime})}^s $, but   $ {I(D)}^{(s)} = {I(D)}^s $ for some $s > 1$.  For example consider the  weighted oriented  paths $D$ and $D^{\prime}$ as in Figure \ref{fig.4}. Then $I(D) = (x_1x_2^2,x_2x_3,x_3x_4^2,x_5x_4^2,x_6x_5^2)$ and $I(D^{\prime}) = (x_1x_2^2,x_2x_3,x_3x_4^2)$. Note that $D^{\prime}$  is an induced  path of $D$.  Using Macaulay 2, we see that $ {I(D^{\prime})}^{(2)} \neq {I(D^{\prime})}^2 $ and   $ {I(D)}^{(2)} = {I(D)}^2 $. 
	\begin{figure}[!ht]
		\begin{tikzpicture}[scale=0.7]  
			\begin{scope}[ thick, every node/.style={sloped,allow upside down}] 
				\definecolor{ultramarine}{rgb}{0.07, 0.04, 0.56} 
				\definecolor{zaffre}{rgb}{0.0, 0.08, 0.66}
				
				\draw [fill, black](0,0) --node {\midarrow}(1.5,0);  
				\draw [fill, black](1.5,0) --node {\midarrow}(3,0);
				\draw [fill, black](3,0) --node {\midarrow}(4.5,0);
				\draw[fill, black]  (6,0) --node {\midarrow}(4.5,0);
				\draw[fill, black]  (7.5,0) --node {\midarrow}(6,0);

				\draw[fill] [fill] (0,0) circle [radius=0.04];
				\draw[fill] [fill] (1.5,0) circle [radius=0.04];
				\draw[fill] [fill] (3,0) circle [radius=0.04];
				\draw[fill] [fill] (4.5,0) circle [radius=0.04];
				\draw[fill] [fill] (6,0) circle [radius=0.04];
				\draw[fill] [fill] (7.5,0) circle [radius=0.04];

				\node at (0,0.5) {$x_1$};
				\node at (1.5,0.5) {$x_2$};
				\node at (3,0.5) {$x_{3}$};  
				\node at (4.5,0.5) {$x_{4}$};
				\node at (6,0.5) {$x_{5}$};  
				\node at (7.5,0.5) {$x_{6}$};

				\node at (0,-0.5) {$1$};
				\node at (1.5,-0.5) {$2$};
				\node at (3,-0.5) {$1$};  
				\node at (4.5,-0.5) {$2$};
				\node at (6,-0.5) {$2$};     
				\node at (7.5,-0.5) {$1$};

				\draw [fill, black](9.5,0) --node {\midarrow}(11,0);
				\draw [fill, black](11,0) --node {\midarrow}(12.5,0);  
					\draw [fill, black](12.5,0) --node {\midarrow}(14,0);

				\draw[fill] [fill] (9.5,0) circle [radius=0.04];
				\draw[fill] [fill] (11,0) circle [radius=0.04];
				\draw[fill] [fill] (12.5,0) circle [radius=0.04];
				\draw[fill] [fill] (14,0) circle [radius=0.04];

				\node at (9.5,0.5) {$x_1$};
				\node at (11,0.5) {$x_{2}$};  
				\node at (12.5,0.5) {$x_{3}$};
				\node at (14,0.5) {$x_{4}$};
				
				\node at (9.5,-0.5) {$1$};
				\node at (11,-0.5) {$2$};  
				\node at (12.5,-0.5) {$1$};
				\node at (14,-0.5) {$2$};

				\node at (3,-1.2) {$D$};
				\node at (11,-1.2) {$D^{\prime}$};	

			\end{scope}
		\end{tikzpicture}
		\caption{A weighted naturally oriented path $D$ containing an induced weighted oriented path  $D^{\prime}$.}\label{fig.4}
	\end{figure}

\end{remark}

If $ D^{\prime} $ is an induced digraph of $D$ and $ C $ is a maximal strong vertex cover of $ D $,  we see that $ C \cap V(D^{\prime}) $ may not contain any maximal strong vertex cover of  $D^{\prime}$  in the next example. 

\begin{example}
Consider the  weighted oriented paths $D$ and $D^{\prime}$ as in Figure \ref{fig.3}. Note that $D^{\prime}$  is an induced path of $D$. Using Macaulay $2$,  the strong vertex covers of $ D $ are $ \{x_1, x_3, x_5\},$ $ \{x_2, x_3, x_5\},$ $   \{x_2,x_4,x_5\},$ $   \{x_2, x_4, x_6\},$ $ \{x_1, x_3, x_4, x_6\},$ $ \{x_2, x_3, x_4, x_6\}  $ and  the strong vertex covers of  $D^{\prime}$ are $\{x_2\},$ $ \{x_1, x_3\},$ $ \{x_2, x_3\}$. 
Note that $ C = \{x_2, x_4, x_5\} $ is a maximal strong vertex cover of $ D $,
but $ C \cap V(D^{\prime})= \{x_2\} $ does not contain any maximal strong vertex cover of  $D^{\prime}$.	
	
 \begin{figure}[!ht]
		\begin{tikzpicture}[scale=0.7]  
		\begin{scope}[ thick, every node/.style={sloped,allow upside down}] 
		\definecolor{ultramarine}{rgb}{0.07, 0.04, 0.56} 
		\definecolor{zaffre}{rgb}{0.0, 0.08, 0.66}
		
		\draw [fill, black](0,0) --node {\midarrow}(1.5,0);  
		\draw [fill, black](1.5,0) --node {\midarrow}(3,0);
		\draw [fill, black](3,0) --node {\midarrow}(4.5,0);
		\draw[fill, black]  (4.5,0) --node {\midarrow}(6,0);
		\draw[fill, black]  (6,0) --node {\midarrow}(7.5,0);

		\draw[fill] [fill] (0,0) circle [radius=0.04];
		\draw[fill] [fill] (1.5,0) circle [radius=0.04];
		\draw[fill] [fill] (3,0) circle [radius=0.04];
		\draw[fill] [fill] (4.5,0) circle [radius=0.04];
		\draw[fill] [fill] (6,0) circle [radius=0.04];
		\draw[fill] [fill] (7.5,0) circle [radius=0.04];

		\node at (0,0.5) {$x_1$};
		\node at (1.5,0.5) {$x_2$};
		\node at (3,0.5) {$x_{3}$};  
		\node at (4.5,0.5) {$x_{4}$};
		\node at (6,0.5) {$x_{5}$};  
		\node at (7.5,0.5) {$x_{6}$};

		\node at (0,-0.5) {$1$};
		\node at (1.5,-0.5) {$7$};
		\node at (3,-0.5) {$1$};  
		\node at (4.5,-0.5) {$1$};
		\node at (6,-0.5) {$1$};
		\node at (7.5,-0.5) {$1$};

		\draw [fill, black](9.5,0) --node {\midarrow}(11,0);
		\draw [fill, black](11,0) --node {\midarrow}(12.5,0);

		\draw[fill] [fill] (9.5,0) circle [radius=0.04];
		\draw[fill] [fill] (11,0) circle [radius=0.04];
		\draw[fill] [fill] (12.5,0) circle [radius=0.04];

		\node at (9.5,0.5) {$x_1$};
		\node at (11,0.5) {$x_{2}$};  
		\node at (12.5,0.5) {$x_{3}$};

		\node at (9.5,-0.5) {$1$};
		\node at (11,-0.5) {$7$};  
		\node at (12.5,-0.5) {$1$};

		\node at (4,-1.2) {$D$};
		\node at (11,-1.2) {$D^{\prime}$};	

		\end{scope}
		\end{tikzpicture}
		\caption{A weighted naturally oriented path $D$ containing an induced weighted oriented path  $D^{\prime}$.}\label{fig.3}
	\end{figure}

\end{example}
\vspace*{0.5cm}

\begin{theorem}\label{induced}
	Let $ D $ be a  weighted oriented graph.  Let $ D^{\prime} $ be an induced digraph of $D$ and it satisfies the condition  ``if $C$ is  a maximal strong vertex cover of $D$, then every strong vertex cover of $ D^{\prime} $ contained in $C$, is subset of at most one maximal strong vertex cover of $ D^{\prime} $  contained in $C$".   	
 If $ {I(D^{\prime})}^{(s)} \neq {I(D^{\prime})}^s $ for some $s \geq 2,$ then  $ {I(D)}^{(s)} \neq {I(D)}^s $.

\end{theorem}

\begin{proof}
	
Let $I=I(D)$ and $\tilde{I}= I(D^{\prime})$. Since
$ I(D^{\prime})^s $
is the restriction of $ I(D)^s $ to $ D^{\prime} $,
it suffices to show that $\tilde{I}^{(s)} \subseteq  I^{(s)}.$  
Equivalently, by Lemma \ref{cooper}, it suffices to show that ``if $ f $ is a minimal generator of
$ \tilde{I}^{(s)} $,
then $f \in I_{\subseteq C}^s $ for each maximal strong vertex cover $ C $ of $ D $.”

Let $C$ be a maximal strong  vertex cover of $D$.  If $C$ contains two maximal  strong vertex covers  of $D^{\prime}$,  then those two maximal strong vertex covers can not be subset of one   maximal strong vertex cover of $ D^{\prime} $  contained in $C$, by our assumption.  Therefore  the proof follows from the following two cases.

\textbf{Case-I}:   $C$ contains exactly one maximal  strong vertex cover  of $D^{\prime}$.    

Let  ${\tilde{C}}$ be the  maximal  strong vertex cover  of $D^{\prime}$ contained in $C$. First we claim that $ \tilde{I}_{\subseteq \tilde{C}} \subseteq  I_{\subseteq C} $.

\textbf{Case (1)} Let $x_i \in L_1^{D^{\prime}}({\tilde{C}})$.    

$\textbf{\underline{\mbox{Case  (1.a)}}}$ Assume that $ w_i = 1  $ in $D^{\prime}$.

$\textbf{\underline{\mbox{Case  (1.a.i)}}}$ Assume $ w_i = 1  $ in $D$.
We claim $x_i \in L_1^{D}(C)  \cup  L_2^{D}(C)$. Suppose $x_i \in  L_3^{D}(C).$  Since $x_i \in L_1^{D^{\prime}}({\tilde{C}})$,  $ N_{D^{\prime}}^+ (x_i) \cap {\tilde{C}}^c \neq        \phi $ and so  $ N_{D^{\prime}} (x_i) \cap {\tilde{C}}^c \neq        \phi $.  Let $N_{D^{\prime}} (x_i) \cap {\tilde{C}}^c = \{x_{j_1},x_{j_2},\ldots,x_{j_q}\}$. Note that  $N_{D^{\prime}}(x_i) \subseteq  N_D(x_i)$ and by Lemma \ref{L3}, $N_D(x_i) \subset C$. Then $\{x_{j_1},x_{j_2},\ldots,x_{j_q}\} \subset C$.  Let $C_1  = \tilde{C}\cup  \{x_{j_1},x_{j_2},\ldots,x_{j_q}\} $. Since $ N_{D^{\prime}}[x_i]  \subset C_1$, $C_1 \setminus \{x_i\}$ is a vertex cover of $D^{\prime}$. Observe that $x_{j_1} \notin  L_3^{D^{\prime}}(C_1 \setminus \{x_i\})$ and hence by Corollary \ref{svc3}, there exists a strong vertex cover  $\tilde{\tilde{C}} \subseteq [C_1 \setminus \{x_i\}] $  of $D^{\prime}$ such that $x_{j_1}  \in  \tilde{\tilde{C}}$.    Here $ \tilde{\tilde{C}} \subset C  $, $x_i \in \tilde{C}   $, $x_i \notin \tilde{\tilde{C}}   $, $x_{j_1} \in \tilde{\tilde{C}}   $ and $x_{j_1} \notin {\tilde{C}}$.
  Since   $\tilde{C}$ is maximal, $\tilde{C}$ and  $\tilde{\tilde{C}}$ can not be subsets of at most one maximal strong vertex cover  of $ D^{\prime} $ contained in $C$.  So it contradicts our assumption. Thus the claim follows.

$\textbf{\underline{\mbox{Case  (1.a.ii)}}}$ Assume $ w_i \neq 1  $ in $D$. That means $N_D^-(x_i)  \neq \phi$. Since $ w_i = 1  $ in $D^{\prime}$, $N_{D^{\prime}}^-(x_i)  = \phi$.  
We claim $x_i \in L_1^{D}(C)  $. Suppose $x_i \in    L_2^{D}(C) \cup L_3^{D}(C).$
Here $ N_{D^{\prime}}^+ (x_i) \cap {\tilde{C}}^c \neq        \phi $.  Let $N_{D^{\prime}}^+ (x_i) \cap {\tilde{C}}^c = \{x_{j_1},x_{j_2},\ldots,x_{j_r}\}$
and we know $N_{D^{\prime}}^- (x_i) =\phi  $.  Note that  $N^+_{D^{\prime}}(x_i) \subseteq  N^+_D(x_i)$, and  $N^+_D(x_i) \subset C$ because $x_i \notin  L_1^{D}(C) .$
Then $\{x_{j_1},x_{j_2},\ldots,x_{j_r}\} \subset C$.   Let $C_1  = \tilde{C}\cup  \{x_{j_1},x_{j_2},\ldots,x_{j_r}\} $. Since $ N_{D^{\prime}}[x_i]  \subset C_1$, $C_1 \setminus \{x_i\}$ is a vertex cover of $D^{\prime}$. By the similar argument as in  ${{\mbox{Case  (1.a.i)}}}$,  every strong vertex cover of $ D^{\prime} $ contained in $C$, is not  subset of  at most one maximal  strong vertex cover of $ D^{\prime} $  contained in $C$, which contradicts our assumption. Thus the claim follows.

$\textbf{\underline{\mbox{Case  (1.b)}}}$ Assume that $ w_i \neq 1 $ in $D^{\prime}$. We claim $x_i \in L_1^{D}(C)$. Suppose $x_i \in L_2^{D}(C) \cup  L_3^{D}(C).$ Here $ N_{D^{\prime}}^+ (x_i) \cap {\tilde{C}}^c \neq        \phi $. Let $N_{D^{\prime}}^+ (x_i) \cap {\tilde{C}}^c = \{x_{j_1},x_{j_2},\ldots,x_{j_r}\}$. Note that $\{x_{j_1},x_{j_2},\ldots,x_{j_r}\} \subset C$ because $x_i \notin  L_1^{D}(C) $. 
Let $C_1  = \tilde{C}  \cup  \{x_{j_1},x_{j_2},\ldots,x_{j_r}\} $. Then $x_i \notin L_1^{D^{\prime}}({{C_1}})$ and by Lemma \ref{L3}, we have  $L_3^{D^{\prime}}({{\tilde{C}}}) \cup \{x_{j_1},x_{j_2},\ldots,x_{j_r}\} \subseteq L_3^{D^{\prime}}(C_1)$.       Let $x_l \in L_3^{D^{\prime}}({{\tilde{C}}})$. Then $ x_l $ satisfies SVC condition
on  $ \tilde{C} $ because $ \tilde{C} $ is strong.  By Lemma \ref{svc1},  $x_l \in L_3^{D^{\prime}}(C_1)$  and  $ x_l $ satisfies SVC condition
on  $ C_1 $.  Since $ x_i \in  N_{D^{\prime}}^- (x_{j_t}) \cap   V^+(D^{\prime})\cap  [L_2^{D^{\prime}}(C_1) \cup  L_3^{D^{\prime}}(C_1)]$,  $x_{j_t}$
satisfies SVC condition on $ C_1 $ for $1 \leq t \leq r$.  
If $L_3^{D^{\prime}}({{\tilde{C}}}) \cup \{x_{j_1},x_{j_2},\ldots,x_{j_r}\}= L_3^{D^{\prime}}(C_1)$,  then each element of $L_3^{D^{\prime}}({{C_1}})$  satisfies SVC condition on $ C_1 $. So $ C_1 $ is a strong
vertex cover of $D^{\prime}$. But it contradicts the fact that $ \tilde{C} $  is one  maximal  strong vertex cover of $D^{\prime}$.  Now  we assume that $L_3^{D^{\prime}}({{\tilde{C}}}) \cup \{x_{j_1},x_{j_2},\ldots,x_{j_r}\} \subsetneq L_3^{D^{\prime}}({{C_1}})$. Let $x_k  \in   L_3^{D^{\prime}}({{C_1}}) \setminus [L_3^{D^{\prime}} ({{\tilde{C}}}) \cup \{x_{j_1},x_{j_2},\ldots,x_{j_r}\} ]$. That means $ x_k $ lies in the neighbourhood of  $ x_{j_t} $ for
some $ t \in [r]$.  Without loss of generality let $x_k \in N_{D^{\prime}} (x_{j_1})$.   By Lemma \ref{L3}, $C_1 \setminus \{x_k\}$ is a vertex cover of $ D^{\prime} $. Then $x_{j_1} \notin    L_3^{D^{\prime}}({{C_1 \setminus \{x_k\}}})$ and so by Corollary \ref{svc3}, there exists a strong vertex cover  $\tilde{\tilde{C}} \subseteq [C_1 \setminus \{x_k\}]  $  of $D^{\prime}$ such that $x_{j_1} \in  \tilde{\tilde{C}}$. Here $x_k \in \tilde{C}   $, $x_k \notin \tilde{\tilde{C}}   $, $x_{j_1} \in \tilde{\tilde{C}}   $ and $x_{j_1} \notin {\tilde{C}}.$  Then by the  same  argument as in ${{\mbox{Case  (1.a.i)}}}$,  it contradicts our assumption. Thus the claim follows.

\textbf{Case (2)} Let $x_i \in L_2^{D^{\prime}}({\tilde{C}})$.

   We claim $x_i \in L_1^{D}(C)  \cup  L_2^{D}(C)$. Suppose $x_i \in  L_3^{D}(C).$
Since $x_i \in L_2^{D^{\prime}}({\tilde{C}})$, we have $ N_{D^{\prime}}^+ (x_i) \cap {\tilde{C}}^c =        \phi $ and  $ N_{D^{\prime}}^- (x_i) \cap {\tilde{C}}^c \neq        \phi $.  Let $N_{D^{\prime}}^- (x_i) \cap {\tilde{C}}^c = \{x_{j_1},x_{j_2},\ldots,x_{j_r}\}$.
Note that $\{x_{j_1},x_{j_2},\ldots,x_{j_r}\} \subset C$ because $x_i \in  L_3^{D}(C).$  Let $C_1  = \tilde{C}\cup  \{x_{j_1},x_{j_2},\ldots,x_{j_r}\} $. Since $ N_{D^{\prime}}[x_i]  \subset C_1$, $C_1 \setminus \{x_i\}$ is a vertex cover of $D^{\prime}$. By the similar argument as in  ${{\mbox{Case  (1.a.i)}}}$,  we get a contradiction. Thus  $x_i \in L_1^{D}(C)  \cup  L_2^{D}(C)$. 

\textbf{Case (3)} Let $x_i \in L_3^{D^{\prime}}({\tilde{C}})$.    Then  $x_i \in L_1^{D}(C)  \cup  L_2^{D}(C)  \cup  L_3^{D}(C)$.

Here $ \tilde{I}_{\subseteq \tilde{C}}  = (L_1^{D^{\prime}}(\tilde{C}))  + (x_i^{w_i} | x_i \in  L_2^{D^{\prime}}(\tilde{C})     ) + (x_ix_j^{w_j} ~|~ (x_i,x_j) \in E({D^{\prime}}), x_i \in 
L_2^{{D^{\prime}}}(\tilde{C}) \cup 	L_3^{{D^{\prime}}}(\tilde{C}) ~\mbox{and}~ x_j \in 
L_3^{{D^{\prime}}}(\tilde{C})  )   $  and  $ I_{\subseteq C}  = (L_1^D(C))  + (x_i^{w_i} ~|~ x_i \in  L_2^D(C)     ) + (x_ix_j^{w_j} ~|~ (x_i,x_j) \in E(D),~ x_i \in 
L_2^{D}(C) \cup 	L_3^{D}(C) ~\mbox{and}~ x_j \in 
L_3^{D}(C)  )   $.  Hence $ \tilde{I}_{\subseteq \tilde{C}} \subseteq  I_{\subseteq C} $.  
\vspace*{0.2cm}\\
If $f \in \mathcal{G}(\tilde{I}^{(s)})$, by Lemma \ref{cooper}, we have $f \in   \tilde{I}^s_{\subseteq \tilde{C}}  $, and so $f \in    I^s_{\subseteq C} $.
\vspace*{0.2cm}\\
\textbf{Case-II}:  $C$ contains no maximal  strong vertex cover  of $D^{\prime}$.

Since $ D^{\prime} $ is an induced digraph of $D$,  $C \cap V(D^{\prime}) $ is a  vertex cover of $D^{\prime}.$ By Lemma \ref{svc2}, there exists a strong vertex cover  ${C^\prime} \subseteq [C \cap V(D^{\prime})]$  of $D^{\prime}$. By our assumption, ${C^\prime}$  is not a maximal strong vertex cover of $D^{\prime}$. Thus there exists a maximal strong vertex cover ${\tilde{C}}$ of $D^{\prime}$ such that  ${C^\prime}  \subsetneq {\tilde{C}}$. Again by our assumption, $ {\tilde{C}}  \nsubseteq C$. We claim  $ \tilde{I}_{\subseteq \tilde{C}} \subseteq  I_{\subseteq C} $.

Consider one element $x  \in \tilde{C} \setminus C $. Since $x \notin C$, we have $x \notin  C^\prime  $.  This implies $N_{D^{\prime}}(x)  \subset C^\prime$ and so   $N_{D^{\prime}}(x)  \subset {\tilde{C}}$ because ${C^\prime}  \subset {\tilde{C}}$. Thus by Lemma \ref{L3}, $x\in L_3^{D^{\prime}}({\tilde{C}})$. Hence if $x\in L_1^{D^{\prime}}({\tilde{C}}) \cup L_2^{D^{\prime}}({\tilde{C}})$, then $x \in C$.

\textbf{Case (1)} Let $x_i \in L_1^{D^{\prime}}({\tilde{C}})$. Then $x_i \in C$.

$\textbf{\underline{\mbox{Case  (1.a)}}}$ Assume that $ w_i = 1  $ in $D^{\prime}$.

$\textbf{\underline{\mbox{Case  (1.a.i)}}}$ Assume $ w_i = 1  $ in $D$.
We claim $x_i \in L_1^{D}(C)  \cup  L_2^{D}(C)$. Suppose $x_i \in  L_3^{D}(C).$   Since  ${C^\prime}  \subset {\tilde{C}}$,  $x_i \in L_1^{D^{\prime}}({\tilde{C}})$  implies  $x_i \in L_1^{D^{\prime}}({{C^\prime}})$.
Note that $ N_{D^{\prime}} (x_i) \cap {\tilde{C}}^c \neq        \phi $  and     $ N_{D^{\prime}} (x_i) \cap {\tilde{C}}^c  \subseteq  N_{D^{\prime}} (x_i) \cap {{C^\prime}}^c  $.  Let $N_{D^{\prime}} (x_i) \cap {{C^\prime}}^c  = \{x_{j_{1}},x_{j_{2}},\ldots,x_{j_q}\}$. Without loss of generality we can assume that $ x_{j_{1}}  \in  N_{D^{\prime}} (x_i) \cap {\tilde{C}}^c   $.
 Since $x_i \in  L_3^{D}(C)$, we have  $\{x_{j_1},x_{j_2},\ldots,x_{j_q}\} \subset C$. Let $C_1  = {C^\prime}  \cup  \{x_{j_1},x_{j_2},\ldots,x_{j_q}\} $.    Since $ N_{D^{\prime}}[x_i]  \subset C_1$, $C_1 \setminus \{x_i\}$ is a vertex cover of $D^{\prime}$. Observe that $x_{j_1} \notin  L_3^{D^{\prime}}(C_1 \setminus \{x_i\})$ and hence by Corollary \ref{svc3}, there exists a strong vertex cover  $\tilde{\tilde{C}} \subseteq [C_1 \setminus \{x_i\}] $  of $D^{\prime}$ such that $x_{j_1}  \in  \tilde{\tilde{C}}$.  Here ${C^\prime}  \subset {\tilde{C}}$,  $ \tilde{\tilde{C}} \subset C  $,  $x_i \in {C^\prime}   $, $x_i \notin \tilde{\tilde{C}}   $, $x_{j_1} \in \tilde{\tilde{C}}   $ and $x_{j_1} \notin {C^\prime}$. If ${C^\prime}  $ and  $ \tilde{\tilde{C}}   $ are  subset of one maximal strong vertex cover $ C^{\prime\prime} $ of $D^{\prime}$ contained in $C$, then $C^{\prime}  \subsetneq C^{\prime\prime}$ and it is a contradiction by Lemma \ref{svc2}.  
Thus the claim follows.

$\textbf{\underline{\mbox{Case  (1.a.ii)}}}$ Assume $ w_i \neq 1  $ in $D$.   
We claim $x_i \in L_1^{D}(C)  $. Suppose $x_i \in    L_2^{D}(C) \cup L_3^{D}(C).$
Note that $x_i \in L_1^{D^{\prime}}({{C^\prime}})$. Then by the similar argument as in (1.a.i) of Case-II and  Case (1.a.ii) of Case-I, our claim follows.

$\textbf{\underline{\mbox{Case  (1.b)}}}$ Assume that $ w_i \neq 1 $ in $D^{\prime}$. We claim $x_i \in L_1^{D}(C)$. Suppose $x_i \in L_2^{D}(C) \cup  L_3^{D}(C).$  Note that $x_i \in L_1^{D^{\prime}}({{C^\prime}})$. Then by the similar argument as in (1.a.i) of Case-II and  Case (1.b) of Case-I, our claim follows.

\textbf{Case (2)} Let $x_i \in L_2^{D^{\prime}}({\tilde{C}})$. Then $x_i \in C$.

We claim $x_i \in L_1^{D}(C)  \cup  L_2^{D}(C)$. Suppose $x_i \in  L_3^{D}(C).$
Since ${C^\prime}  \subset {\tilde{C}}$, $x_i \in L_1^{D^{\prime}}({{C^\prime}})  \cup L_2^{D^{\prime}}({{C^\prime}}) $. If $x_i \in L_1^{D^{\prime}}({{C^\prime}})  $,
by the similar argument as in Case  (1.a.i) of Case-II, we get a contradiction.  If $x_i \in L_2^{D^{\prime}}({{C^\prime}})  $,
by the similar argument as in Case  (1.a.i) of Case-II and  Case (2) of Case-I, we get a contradiction. Thus the claim follows.

\textbf{Case (3)} Let $x_i \in L_3^{D^{\prime}}({\tilde{C}})$.    Then  $x_i \in L_1^{D}(C)  \cup  L_2^{D}(C)  \cup  L_3^{D}(C)$.

By the similar argument as in Case-I, we have  $ \tilde{I}_{\subseteq \tilde{C}} \subseteq  I_{\subseteq C} $.
\vspace*{0.2cm}\\
If $f \in \mathcal{G}(\tilde{I}^{(s)})$, by Lemma \ref{cooper}, we have $f \in   \tilde{I}^s_{\subseteq \tilde{C}}  $, and so $f \in    I^s_{\subseteq C} $.
\end{proof}

\begin{remark}
	The above theorem may not be true if we remove the given condition.
	
\hspace*{0.5cm}For example consider the  weighted oriented  paths $D$ and $D^{\prime}$ as in Figure \ref{fig.3}. Here $I(D) = (x_1x_2^3,x_2x_3,x_3x_4,x_4^3x_5)$ and $I(D^{\prime}) = (x_1x_2^3,x_2x_3,x_3x_4)$.  
Using Macaulay 2, the strong vertex covers of $D$ are $\{x_2,x_4\},$  $\{x_1,x_3,x_4\},$ $\{x_1,x_3,x_5\},$ $\{x_2,x_3,x_4\}$, $\{x_2,x_3,x_5\}$ and the strong vertex covers of $D^{\prime}$ are $\{x_1,x_3\},$  $\{x_2,x_3\}$,  $\{x_2,x_4\}$.  Let  $C =\{x_2,x_3,x_4\}$.
Note that $C$  contains  the two maximal strong vertex covers $\{x_2,x_3\}$ and $\{x_2,x_4\}$ of $D^{\prime}$. Those two maximal strong vertex covers are not  subset of one  strong vertex cover of $ D^{\prime} $  contained in $C$.  Using Macaulay 2, we see that $I(D^{\prime})^{(2)} \neq I(D^{\prime})^2$ but $I(D)^{(2)} = I(D)^2$.

\end{remark}

Next  we see some applications of  Theorem \ref{induced} for induced weighted oriented  paths.

\begin{theorem}\label{inducedpath} Let $ D $ be a weighted oriented path with $ V(D) = \{x_1, x_2, \ldots , x_n, y_1, y_2,\\ \ldots , y_m\}. $ Let $ D^{\prime} $ be the induced weighted oriented path of $D$ with $ V(D^{\prime}) = \{x_1, x_2,  \ldots , x_n\} $ and $N_{D}^- (x_{n-1}) \cap V^+(D) = \phi.$ If $ {I(D^{\prime})}^{(s)} \neq {I(D^{\prime})}^s $ for some $s \geq 2,$ then $ {I(D)}^{(s)} \neq {I(D)}^s $.
\end{theorem}

\begin{proof}
	Let $ C $ be a maximal strong vertex cover of $ D.  $  We claim  every strong vertex cover of $ D^{\prime} $ contained in $C$, is subset of at most one maximal strong vertex cover of $ D^{\prime} $  contained in $C$. Since  $N_{D}^- (x_{n-1}) \cap V^+(D) = \phi$, by Remark \ref{s.v.1}, we have $x_{n-1} \notin L_3^{D}(C)$. Then by Lemma \ref{L3},    
  one of $x_{n-2},$ $x_{n-1}$  and  $x_{n} \notin C$. Therefore we consider the following three cases.

	\textbf{Case (1)} Assume that $x_n \notin C$. We claim $ C_1 =    C \cap V(D^{\prime}) $ is a strong vertex
	cover of $  D^{\prime} $. Note that $ C_1  $ is a  vertex
	cover of $  D^{\prime} $. Since $x_n \notin C_1$,  $x_{n-1} \in C_1$. Let $x_i \in L_3^{D^{\prime}}({{C_1}})$.
	Then  by Lemma \ref{L3}, $x_i \in \{x_1, x_2, \ldots , x_{n-2}\}$ and   $x_i \in L_3^{D}({C})$.
	Since $ C $ is strong  in $D$,  there exists  some  $x_j \in  N_{D}^- (x_{i}) \cap V^+(D)\cap  [L_2^{D}(C) \cup  L_3^{D}({C})]$.   Note that  orientations of edges from $x_1$ to $x_{n}$  and  weights  of  vertices from $x_1$ to $x_{n-1}$  are same in both the paths $D$  and  $D^{\prime}$. Here    $x_j \notin L_1^{D}(C)$ implies $ N_{D}^+ (x_j)\cap {{C}}^c = \phi $   and $x_j \in \{x_1, x_2, \ldots , x_{n-1}\}$.  Since  $C_1 =    C \cap V(D^{\prime})$, we have $ N_{D^{\prime}}^+ (x_j)\cap {{C_1}}^c = N_{D}^+ (x_j)\cap {{C}}^c = \phi $. Thus $x_j \notin L_1^{D^{\prime}}(C_1)$.
	So  
	$x_j \in N_{D^{\prime}}^- (x_{i}) \cap V^+(D^{\prime})\cap  [L_2^{D^{\prime}}({{C_1}}) \cup  L_3^{D^{\prime}}({{C_1}})]$. Hence $ C_1 $ is a strong vertex cover of $D^{\prime}$ and it contains  each strong vertex cover of $ D^{\prime} $ contained in $C$.    Therefore 
	every strong vertex cover of $ D^{\prime} $ contained in $C$, is subset of one   strong vertex cover of $ D^{\prime} $  contained in $C$.

	\textbf{Case (2)}	Assume  that $x_{n-1} \notin C$. We claim $ C_1 =    C \cap V(D^{\prime}) $ is a strong vertex
	cover of $  D^{\prime} $. Note that $ C_1  $ is a  vertex
	cover of $  D^{\prime} $. Here $x_{n-1} \notin C_1$.  This implies $x_{n-2}$ and $x_{n} \in C_1$. Let $x_i \in L_3^{D^{\prime}}({{C_1}})$. By Lemma \ref{L3}, $x_i \in \{x_1, x_2,  \ldots , x_{n-3}\}$.
	Then  by the same
	argument as in Case (1),   our claim follows and every strong vertex cover of $ D^{\prime} $ contained in $C$, is subset of one   strong vertex cover of $ D^{\prime} $ contained in $C$.

	\textbf{Case (3)}	Assume that $x_{n-2} \notin C$. If $x_n \notin C$, then by Case (1),   every strong vertex cover of $ D^{\prime} $ contained in $C$, is subset of one  strong vertex cover of $ D^{\prime} $ contained in $C$. Now we assume  $x_n \in C$.
	
	$\textbf{\underline{\mbox{Case  (3.a)}}}$  Suppose $N_{D^{\prime}}^- (x_{n}) \cap V^+(D^{\prime})\cap  [L_2^{D^{\prime}}({{C_1}}) \cup  L_3^{D^{\prime}}({{C_1}})] \neq \phi$. We claim $ C_1 =    C \cap V(D^{\prime}) $ is a strong vertex
	cover of $  D^{\prime} $. Note that $ C_1  $ is a  vertex
	cover of $  D^{\prime} $. Here $x_{n-2} \notin C_1$.  This implies $x_{n-3}$ and  $x_{n-1} \in C_1$. Let $x_i \in L_3^{D^{\prime}}(C_1)$. By Lemma \ref{L3}, $x_i \in \{x_1, x_2,  \ldots , x_{n-4},x_n\}$.  If  $x_i \in \{x_1, x_2,  \ldots , x_{n-4}\}$, then by the similar
	argument as in Case (1),  $x_i$ satisfies SVC condition on $C_1$. Also $x_n \in L_3^{D^{\prime}}(C_1)$ satisfies SVC condition on $C_1$, by our assumption. Hence $ C_1 $ is strong in $D^{\prime}$ and    every strong vertex cover of $ D^{\prime} $ contained in $C$, is subset of one   strong vertex cover of $ D^{\prime} $ contained in $C$.
	
	$\textbf{\underline{\mbox{Case  (3.b)}}}$ Suppose $N_{D^{\prime}}^- (x_{n}) \cap V^+(D^{\prime})\cap  [L_2^{D^{\prime}}({{C_1}}) \cup  L_3^{D^{\prime}}({{C_1}})] = \phi$.  Let $ C^{\prime}$  is a strong vertex  cover  of  $ D^{\prime} $  contained in $C$.  Suppose  $x_n \in C^{\prime}$.
	Since    $x_{n-2} \notin C^{\prime}$,
	we have  $  x_{n-1} \in C^{\prime}$ and so by Lemma \ref{L3}, $x_n \in L_3^{D^{\prime}}(C^{\prime})$. Since $C^{\prime}  $  is strong,  $x_n$ satisfies SVC condition on $C^{\prime}$.  By Lemma \ref{svc1},  $x_n \in L_3^{D^{\prime}}(C_1)$ and it satisfies SVC condition on $C_1$, i.e.,
 $N_{D^{\prime}}^- (x_{n}) \cap V^+(D^{\prime})\cap  [L_2^{D^{\prime}}({{C_1}}) \cup  L_3^{D^{\prime}}({{C_1}})] \neq \phi$. But it  contradicts our assumption. Therefore   $x_n \notin C^{\prime}$. Hence we can say that    $x_n$ does not belong to any  strong vertex  cover  of  $ D^{\prime} $  contained in $C$.  Let $C_2 = [C \cap V(D^{\prime})] \setminus \{x_n\}$. We claim that $ C_2 $ is a strong vertex cover of $  D^{\prime} $. Since $ [C \cap V(D^{\prime})]  $ is a  vertex
 cover of $  D^{\prime} $ and $ x_{n-2} \notin  [C \cap V(D^{\prime})]  $, we have $ x_{n-1} \in  [C \cap V(D^{\prime})]  $ and    so $ C_2 $ is a  vertex cover of $  D^{\prime} $.  Here $x_{n-2} \notin C_2$ implies $x_{n-3}$ and  $x_{n-1} \in C_2$.  Let $x_i \in L_3^{D^{\prime}}({{C_2}})$. By Lemma \ref{L3},  $x_i \in \{x_1, x_2,  \ldots , x_{n-4}\}$.  By the similar
 argument as in Case (1), we can show $ C_2 $ is  strong and  every strong vertex cover of $ D^{\prime} $ contained in $C$, is subset of one   strong vertex cover of $ D^{\prime} $ contained in $C$.

In all the cases, our claim follows  and the proof follows from Theorem \ref{induced}.        
\end{proof}

\begin{corollary}\label{inducedpath2} Let $ D $ be a weighted oriented path with $ V(D) = \{x_1, x_2,  \ldots , x_n, y_1, y_2,  \ldots ,\\ y_m\}. $ Let $ D^{\prime} $ be the induced weighted oriented path of $D$ with $ V(D^{\prime}) = \{y_1, y_2,  \ldots , y_m\} $ and $N_{D}^- (y_{2}) \cap V^+(D) = \phi.$ If $ {I(D^{\prime})}^{(s)} \neq {I(D^{\prime})}^s $ for some $s \geq 2,$ then $ {I(D)}^{(s)} \neq {I(D)}^s $.
\end{corollary}
\begin{proof}
It follows by the similar argument as in Theorem \ref{inducedpath}.
\end{proof}

\begin{corollary}\label{inducedpath3} Let $ D $ be a weighted oriented path with $ V(D) = \{x_1, x_2,  \ldots , x_n, y_1, y_2,  \ldots ,\\ y_m,z_1,z_2,\ldots,z_l\}.$  Let $ D^{\prime} $ be the induced weighted oriented path of $D$ with $ V(D^{\prime}) = \{y_1, y_2,  \ldots , y_m\} $, where $N_{D}^- (y_{2}) \cap V^+(D) = \phi$  and $N_{D}^- (y_{m-1}) \cap V^+(D) = \phi.$ If $ {I(D^{\prime})}^{(s)} \neq {I(D^{\prime})}^s $ for some $s \geq 2,$ then $ {I(D)}^{(s)} \neq {I(D)}^s $.
\end{corollary}
\begin{proof}
Let $ D_1 $ be the induced weighted oriented path of $D$ with $ V(D_1) = \{x_1, x_2,  \ldots , x_n,\\y_1, y_2,  \ldots , y_m\} $. Here  $ D^{\prime} $ is an induced weighted oriented path of $D_1$. 
Assume that  $ {I(D^{\prime})}^{(s)} \neq {I(D^{\prime})}^s $ for some $s \geq 2$. Since $N_{D}^- (y_{2}) \cap V^+(D) = \phi$, by  Corollary \ref{inducedpath2},  $ {I(D_1)}^{(s)} \neq {I(D_1)}^s $. Note that $N_{D}^- (y_{m-1}) \cap V^+(D) = \phi$. Then by Theorem \ref{inducedpath},  $ {I(D_1)}^{(s)} \neq {I(D_1)}^s $  implies $ {I(D)}^{(s)} \neq {I(D)}^s $. 	
\end{proof}

\begin{remark}
 When we try to find the necessary and sufficient condition for the equality of ordinary and symbolic powers of edge ideals of weighted oriented paths, 	using the above results, we can show the inequality of ordinary and symbolic powers of edge ideals of a larger class of weighted oriented paths by studying the  inequality of ordinary and symbolic powers of edge ideals of a smaller class of weighted oriented paths.
 
 \hspace*{0.5cm}For example consider the  weighted oriented path $D^{\prime}$ as in Figure \ref{fig.5}.  Then $I(D^{\prime}) =   (x_1x_2^{w_2},x_2x_3,x_3x_4,x_4x_5)$.
 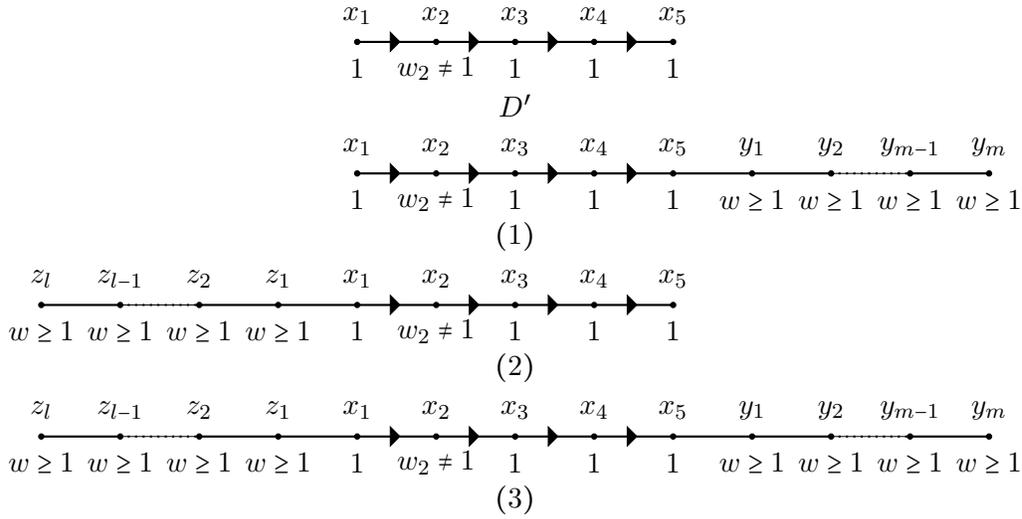
\begin{figure}[!ht]
 	\begin{tikzpicture}[scale=0.7]
 		\begin{scope}[ thick, every node/.style={sloped,allow upside down}] 
 			\definecolor{ultramarine}{rgb}{0.07, 0.04, 0.56} 
 			\definecolor{zaffre}{rgb}{0.0, 0.08, 0.66}
 			
 			\draw [fill, black](0,0) --node {\midarrow}(1.5,0);  
 			\draw [fill, black](1.5,0) --node {\midarrow}(3,0);
 			\draw [fill, black](3,0) --node {\midarrow}(4.5,0);
 			\draw [fill, black](4.5,0) --node {\midarrow}(6,0);

 			\draw[fill] [fill] (0,0) circle [radius=0.04];
 			\draw[fill] [fill] (1.5,0) circle [radius=0.04];
 			\draw[fill] [fill] (3,0) circle [radius=0.04];
 			\draw[fill] [fill] (4.5,0) circle [radius=0.04];
 			\draw[fill] [fill] (6,0) circle [radius=0.04];

 			\node at (0,0.5) {$x_1$};
 			\node at (1.5,0.5) {$x_2$};    
 			\node at (3,0.5) {$x_{3}$};  
 			\node at (4.5,0.5) {$x_{4}$};
 			\node at (6,0.5) {$x_{5}$};

 			\node at (0,-0.5) {$1$};
 			\node at (1.5,-0.5) {$w_2 \neq 1$};
 			\node at (3,-0.5) {$1$};  
 			\node at (4.5,-0.5) {$1$};
 			\node at (6,-0.5) {$1$};

 			\node at (3,-1.2) {$D^{\prime}$};

 			\draw [fill, black](0,-2.5) --node {\midarrow}(1.5,-2.5);  
 			\draw [fill, black](1.5,-2.5) --node {\midarrow}(3,-2.5);
 			\draw [fill, black](3,-2.5) --node {\midarrow}(4.5,-2.5);
 			\draw [fill, black](4.5,-2.5) --node {\midarrow}(6,-2.5);
 			\draw [fill, black](6,-2.5) --(7.5,-2.5);  
 			\draw [fill, black](7.5,-2.5) --(9,-2.5);
 			\draw [fill, black] [dotted,thick] (9,-2.5) -- (10.5,-2.5);
 			\draw [fill, black](10.5,-2.5) -- (12,-2.5);
 			
 			\draw[fill] [fill] (0,-2.5) circle [radius=0.04];
 			\draw[fill] [fill] (1.5,-2.5) circle [radius=0.04];
 			\draw[fill] [fill] (3,-2.5) circle [radius=0.04];
 			\draw[fill] [fill] (4.5,-2.5) circle [radius=0.04];
 			\draw[fill] [fill] (6,-2.5) circle [radius=0.04];
 			\draw[fill] [fill] (7.5,-2.5) circle [radius=0.04];
 			\draw[fill] [fill] (9,-2.5) circle [radius=0.04];
 			\draw[fill] [fill] (10.5,-2.5) circle [radius=0.04];
 			\draw[fill] [fill] (12,-2.5) circle [radius=0.04];
 			
 			\node at (0,-2) {$x_1$};
 			\node at (1.5,-2) {$x_2$};    
 			\node at (3,-2) {$x_{3}$};  
 			\node at (4.5,-2) {$x_{4}$};
 			\node at (6,-2) {$x_{5}$};
 			\node at (7.5,-2) {$y_1$};    
 			\node at (9,-2) {$y_{2}$};  
 			\node at (10.5,-2) {$y_{m-1}$};
 			\node at (12,-2) {$y_{m}$};
 			
 			\node at (0,-3) {$1$};
 			\node at (1.5,-3) {$w_2 \neq 1$};
 			\node at (3,-3) {$1$};  
 			\node at (4.5,-3) {$1$};
 			\node at (6,-3) {$1$};
 			\node at (7.5,-3) {$ w \geq 1$};    
 			\node at (9,-3) {$w \geq 1$};  
 			\node at (10.5,-3) {$w \geq 1$};
 			\node at (12,-3) {$w \geq 1$};
 			
 			\node at (3,-3.7) {$(1)$};

 	 \draw [fill, black](-6,-5) --(-4.5,-5);  
 	\draw [fill, black] [dotted,thick]   (-4.5,-5) -- (-3,-5);
 	\draw [fill, black](-3,-5) --(-1.5,-5);
 	\draw [fill, black](-1.5,-5) --(0,-5);			
 	\draw [fill, black](0,-5) --node {\midarrow}(1.5,-5);  
 	\draw [fill, black](1.5,-5) --node {\midarrow}(3,-5);
 	\draw [fill, black](3,-5) --node {\midarrow}(4.5,-5);
 	\draw [fill, black](4.5,-5) --node {\midarrow}(6,-5);

	\draw[fill] [fill] (-6,-5) circle [radius=0.04];
\draw[fill] [fill] (-4.5,-5) circle [radius=0.04];
\draw[fill] [fill] (-3,-5) circle [radius=0.04];
\draw[fill] [fill] (-1.5,-5) circle [radius=0.04]; 
 	\draw[fill] [fill] (0,-5) circle [radius=0.04];
 	\draw[fill] [fill] (1.5,-5) circle [radius=0.04];
 	\draw[fill] [fill] (3,-5) circle [radius=0.04];
 	\draw[fill] [fill] (4.5,-5) circle [radius=0.04];
 	\draw[fill] [fill] (6,-5) circle [radius=0.04];

 \node at (-6,-4.5) {$z_l$};
 \node at (-4.5,-4.5) {$z_{l-1}$};    
 \node at (-3,-4.5) {$z_{2}$};  
 \node at (-1.5,-4.5) {$z_{1}$};
 	\node at (0,-4.5) {$x_1$};
 	\node at (1.5,-4.5) {$x_2$};    
 	\node at (3,-4.5) {$x_{3}$};  
 	\node at (4.5,-4.5) {$x_{4}$};
 	\node at (6,-4.5) {$x_{5}$};

 	\node at (-6,-5.5) {$w \geq 1$};
 	\node at (-4.5,-5.5) {$w \geq 1$};    
 	\node at (-3,-5.5) {$w \geq 1$};  
 	\node at (-1.5,-5.5) {$w \geq 1$};
 	\node at (0,-5.5) {$1$};
 	\node at (1.5,-5.5) {$w_2 \neq 1$};
 	\node at (3,-5.5) {$1$};  
 	\node at (4.5,-5.5) {$1$};
 	\node at (6,-5.5) {$1$};

 	\node at (3,-6.2) {$(2)$};
 	
 \draw [fill, black](-6,-7.5) --(-4.5,-7.5);  
 \draw [fill, black] [dotted,thick]   (-4.5,-7.5) -- (-3,-7.5);
 \draw [fill, black](-3,-7.5) --(-1.5,-7.5);
 \draw [fill, black](-1.5,-7.5) --(0,-7.5);	
 \draw [fill, black](0,-7.5) --node {\midarrow}(1.5,-7.5);  
 \draw [fill, black](1.5,-7.5) --node {\midarrow}(3,-7.5);
 \draw [fill, black](3,-7.5) --node {\midarrow}(4.5,-7.5);
 \draw [fill, black](4.5,-7.5) --node {\midarrow}(6,-7.5);
 \draw [fill, black](6,-7.5) --(7.5,-7.5);  
 \draw [fill, black](7.5,-7.5) --(9,-7.5);
 \draw [fill, black] [dotted,thick] (9,-7.5) -- (10.5,-7.5);
 \draw [fill, black](10.5,-7.5) -- (12,-7.5);

 \draw[fill] [fill] (-6,-7.5) circle [radius=0.04];
 \draw[fill] [fill] (-4.5,-7.5) circle [radius=0.04];
 \draw[fill] [fill] (-3,-7.5) circle [radius=0.04];
 \draw[fill] [fill] (-1.5,-7.5) circle [radius=0.04];
 \draw[fill] [fill] (0,-7.5) circle [radius=0.04];
 \draw[fill] [fill] (1.5,-7.5) circle [radius=0.04];
 \draw[fill] [fill] (3,-7.5) circle [radius=0.04];
 \draw[fill] [fill] (4.5,-7.5) circle [radius=0.04];
 \draw[fill] [fill] (6,-7.5) circle [radius=0.04];
 \draw[fill] [fill] (7.5,-7.5) circle [radius=0.04];
 \draw[fill] [fill] (9,-7.5) circle [radius=0.04];
 \draw[fill] [fill] (10.5,-7.5) circle [radius=0.04];
 \draw[fill] [fill] (12,-7.5) circle [radius=0.04];

 \node at (-6,-7) {$z_l$};
 \node at (-4.5,-7) {$z_{l-1}$};    
 \node at (-3,-7) {$z_{2}$};  
 \node at (-1.5,-7) {$z_{1}$};
 \node at (0,-7) {$x_1$};
 \node at (1.5,-7) {$x_2$};    
 \node at (3,-7) {$x_{3}$};  
 \node at (4.5,-7) {$x_{4}$};
 \node at (6,-7) {$x_{5}$};
 \node at (7.5,-7) {$y_1$};    
 \node at (9,-7) {$y_{2}$};  
 \node at (10.5,-7) {$y_{m-1}$};
 \node at (12,-7) {$y_{m}$};

 \node at (-6,-8) {$w \geq 1$};
 \node at (-4.5,-8) {$w \geq 1$};    
 \node at (-3,-8) {$w \geq 1$};  
 \node at (-1.5,-8) {$w \geq 1$};
 \node at (0,-8) {$1$};
 \node at (1.5,-8) {$w_2 \neq 1$};
 \node at (3,-8) {$1$};  
 \node at (4.5,-8) {$1$};
 \node at (6,-8) {$1$};
 \node at (7.5,-8) {$ w \geq 1$};    
 \node at (9,-8) {$w \geq 1$};  
 \node at (10.5,-8) {$w \geq 1$};
 \node at (12,-8) {$w \geq 1$};
 
 \node at (3,-8.7) {$(3)$};


 		\end{scope}
 	\end{tikzpicture}
 	\caption{Three classes of weighted oriented paths with common induced weighted oriented path $D^{\prime}$.}\label{fig.5}
\end{figure} 	
By Lemma \ref{atmost}, $I(D^{\prime})^{(3)} \neq I(D^{\prime})^3$.  Now  consider the following three class of paths:

Class (1):  Set of all weighted oriented paths on the vertex set $\{x_1, x_2,  \ldots , x_5, y_1, y_2,  \ldots , y_m\}$ containing the induced weighted oriented path $D^{\prime}$ (as in  Figure \ref{fig.5}),

Class (2):  Set of all weighted oriented paths on the vertex set $\{ z_l, \ldots , z_2 , z_1, x_1, x_2,  \ldots , x_5\}$ containing the induced weighted oriented path $D^{\prime}$ (as in Figure \ref{fig.5}),

Class (3):  Set of all weighted oriented paths on the vertex set $\{ z_l, \ldots , z_2 , z_1, x_1, x_2,  \ldots , x_5,\\ y_1, y_2,  \ldots , y_m\}$ containing the induced weighted oriented path $D^{\prime}$ (as in Figure \ref{fig.5}), 

and in  Figure \ref{fig.5}, where the directions are not mentioned, it can be any direction.

Let $D_1,D_2$  and $D_3$ be any weighted oriented paths of class (1), (2) and (3), respectively.  By  Theorem \ref{inducedpath},    Corollary \ref{inducedpath2}  and  Corollary \ref{inducedpath3}, $ {I(D^{\prime})}^{(3)} \neq {I(D^{\prime})}^3 $   implies    $  {I(D_1)}^{(3)} \neq {I(D_1)}^3 $,  $ {I(D_2)}^{(3)} \neq {I(D_2)}^3 $  and    $ {I(D_3)}^{(3)} \neq {I(D_3)}^3 $, respectively.        
\end{remark}

\section{Symbolic powers of union of two naturally oriented paths  with a common sink vertex}

In this section, we give the necessary and sufficient condition for the equality of ordinary and symbolic powers of edge ideal of union of two naturally oriented paths  with a common sink vertex.

\begin{notation}\label{notation1}
	Let	$ D $ be a weighted oriented  path with
	$ V(D) = \{y_1, y_2, \ldots , y_m, z_1, x_n,   \ldots ,\\  x_2, x_1\}$, $E(D) = \{(y_i,y_{i+1}) ~|~ 1 \leq i \leq m-1 \} \cup \{(y_m,z_{1}), (x_n,z_{1})\} \cup  \{( x_i,  x_{i+1}) ~|~ 1 \leq i \leq n-1 \}$   and $ w(z_1)=1 $ (see Figure \ref{fig.7}).

\begin{figure}[!ht]
	\begin{tikzpicture}[scale=0.8]
		\begin{scope}[ thick, every node/.style={sloped,allow upside down}] 
			\definecolor{ultramarine}{rgb}{0.07, 0.04, 0.56} 
			\definecolor{zaffre}{rgb}{0.0, 0.08, 0.66}
			\draw[fill, black] (-6,0) --node {\midarrow}(-4.5,0);
			\draw[fill, black][dotted,thick] (-3,0)   -- (-4.5,0);
			\draw[fill, black]  (-3,0)  --node {\midarrow}(-1.5,0);
			\draw[fill, black][dotted,thick] (0,0) --(-1.5,0);   
			\draw[fill, black]  (0,0)--node {\midarrow}(1.5,0);
			\draw[fill, black] (1.5,0) --node {\midarrow}(3,0);
			\draw[fill, black] (4.5,0) --node {\midarrow}(3,0);
			\draw[fill, black] (6,0) --node {\midarrow}(4.5,0);
			\draw[fill, black] [dotted,thick](6,0) -- (7.5,0);
			\draw [fill, black](9,0) --node {\midarrow}(7.5,0);
			\draw[fill, black] [dotted,thick] (9,0) --(10.5,0);
			\draw[fill, black]  (12,0) --node {\midarrow}(10.5,0);
			
			\draw[fill] [fill] (-1.5,0) circle [radius=0.04];
			\draw[fill] [fill] (-3,0) circle [radius=0.04];
			\draw[fill] [fill] (-4.5,0) circle [radius=0.04];
			\draw[fill] [fill] (-6,0) circle [radius=0.04];  
			\draw [fill] [fill] (0,0) circle [radius=0.04];
			\draw[fill] [fill] (1.5,0) circle [radius=0.04];
			\draw[fill] [fill] (3,0) circle [radius=0.04];
			\draw[fill] [fill] (4.5,0) circle [radius=0.04];
			\draw[fill] [fill] (6,0) circle [radius=0.04];
			\draw[fill] [fill] (7.5,0) circle [radius=0.04];
			\draw[fill] [fill] (9,0) circle [radius=0.04];
			\draw[fill] [fill] (10.5,0) circle [radius=0.04];
			\draw[fill] [fill] (12,0) circle [radius=0.04];
			
			\node at (-6,0.5) {$y_1$};
			\node at (-4.5,0.5) {$y_{2}$};
			\node at (-3,0.5) {$y_{l}$};
			\node at (-1.5,0.5) {$y_{l+1}$};
			\node at (0,0.5) {$y_{m-1}$};
			\node at (1.5,0.5) {$y_m$};
			\node at (3,0.5) {$z_1$};
			\node at (4.5,0.5) {$x_{n}$};
			\node at (6,0.5) {$x_{n-1}$};
			\node at (7.5,0.5) {$x_{k+1}$};
			\node at (9,0.5) {$x_{k}$};  
			\node at (10.5,0.5) {$x_{2}$};
			\node at (12,0.5) {$x_{1}$};   
		\end{scope}
	\end{tikzpicture}
	\caption{An oriented path which is union of two naturally oriented paths joined at a sink  vertex.}\label{fig.7}
\end{figure}
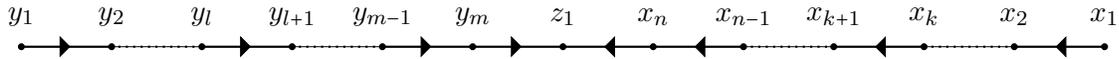

\end{notation}

The following lemma is useful for the proof of Theorem \ref{path5}.
\begin{lemma}\label{pathD1}
	Let	$ D $ be the weighted oriented path same as defined in Notation \ref{notation1}.   Assume that
	there exist two indices $ 1 < l < m $ and $ 1 < k < n $ such that the set of vertices with non-trivial
	weights are precisely $ \{y_l
	,\ldots, y_m\} \cup \{x_k,\ldots, x_n\}. $ Let	$ D_1 $  be the induced path of $D$ with	$ V(D_1) = \{x_1, x_2,  \ldots , x_n\}$ and  $E(D_1) =  \{(x_i,x_{i+1}) ~|~ 1 \leq i \leq n-1 \}$. If $C$ is a maximal strong vertex cover of $ D $,  then prove that $ C \cap V(D_1) $  is a  maximal strong vertex cover of $ D_1 $.

\end{lemma}

\begin{proof}	
Let $C$ be a maximal strong vertex cover of $D$.

Suppose $z_1 \notin C$. We claim $C^{\prime} = C \cup \{z_1\}$ is strong. Since  $z_1 \notin C,$ $N_D(z_1) \subset C$. So by Lemma \ref{L3}, $z_1 \in L^D_3(C^{\prime} )$.  Here  $y_m \in  N_D^{-}(z_1) \cap V^+(D) \cap [{L_2^D{(C^{\prime})}} \cup   {L_3^D{(C^{\prime})}} ] .$ So $z_1$ satisfies the SVC condition on $C^{\prime}$.  If $y_m \in L^D_3(C^{\prime} )$, then $y_{m-1} \in  N_D^{-}(y_m) \cap V^+(D) \cap [{L_2^D{(C^{\prime})}} \cup   {L_3^D{(C^{\prime})}} ] .$ So $y_m$ satisfies the SVC condition on $C^{\prime}$.  If $x_n \in L^D_3(C^{\prime} )$, by the similar argument, $x_n$ satisfies the SVC condition on $C^{\prime}$.  If $v \in L^D_3(C^{\prime} )$, where $v \notin \{y_m,z_1,x_n\}$, then  by Lemma \ref{L3},  
$v \in L^D_3(C)$. Since $C$ is strong, $v$ satisfies the SVC condition on $C$ and so by Lemma \ref{svc1}, $v$ satisfies the SVC condition on $C^{\prime}$. Hence $C^{\prime}  $ is strong. It's a contradiction because $C$ is maximal.  Therefore $z_1 \in C$.

Let $C_1 =  C \cap V(D_1) $.  Since $z_1 \in C$, by Lemma \ref{L3}, $v \in  L^D_3(C_1)$ implies $v \in  L^D_3(C)$.  By the similar argument as in Case (1) of Theorem \ref{inducedpath}, we can show that $C_1$ is  strong in $D_1$.

Suppose $C_1$ is not maximal in $D_1$.  That means there exists a strong vertex cover $C_2$ of $D_1$ such that $C_2= C_1 \sqcup \{ x_{i_1},\ldots,x_{i_m}  \}$, where $ \{ x_{i_1},\ldots,x_{i_m}  \} \subset V(D_1)$  for some $m > 1$.  Now we claim $C^{\prime\prime} = C \sqcup \{ x_{i_1},\ldots,x_{i_m}  \}$ is strong.  By the similar argument as in Case (1) of Theorem \ref{inducedpath}, we can show that $C^{\prime\prime}$ is  strong in $D$. It's a contradiction because $C$ is maximal.  Hence  $C_1$ is a maximal strong vertex cover of $D_1$.
\end{proof}

\begin{theorem}\label{path5}
	Let	$ D $ be the weighted oriented path same as defined in Notation \ref{notation1}.   Assume that
	there exist two indices $ 1 < l < m $ and $ 1 < k < n $ such that the set of vertices with non-trivial
	weights are precisely $ \{y_l
		,\ldots, y_m\} \cup \{x_k,\ldots, x_n\}. $  Then $ I(D)^{(s)} = I(D)^{s} $ for all $s \geq 2$.  
	
\end{theorem}

\begin{proof}	
	 Let	$ D_1 $ and $D_2$ be the induced paths of $D$ with	$ V(D_1) = \{x_1, x_2,  \ldots , x_n\}$,  $E(D_1) =  \{(x_i,x_{i+1}) ~|~ 1 \leq i \leq n-1 \}$, 	$ V(D_2) = \{y_1, y_2,\ldots ,  y_m\}$ and  $E(D_2) =  \{(y_i,y_{i+1}) ~|~ 1 \leq i \leq m-1 \} $.

By \cite[Theorem 3.6]{kanoy}, let  $\{C_{\alpha}\}, \{C_{\beta}\}$ and $\{C_{\gamma}\}  $ be the collections of all maximal strong vertex covers of $ D_1 $, where

$\displaystyle{ C_{\alpha} =  } \bigg\{C^{\prime}_{\alpha} |    C^{\prime}_{\alpha}~\mbox{is a minimal vertex cover of the induced path}~D(x_1, \ldots , x_{k-2})\bigg\} \\ 
\hspace*{0.75cm}\bigcup \{x_{k-1}, x_{k+1}, x_{k+2}, x_{k+3}, \ldots , x_n\}$,
\vspace*{0.2cm}\\
$\displaystyle{ C_{\beta} =  } \bigg\{C^{\prime}_{\beta} |    C^{\prime}_{\beta}~\mbox{is a minimal vertex cover of the induced path}~D(x_1, \ldots , x_{k-3})\bigg\} \\ 
\hspace*{0.75cm}\bigcup \{x_{k-2},x_k, x_{k+1}, x_{k+2},  \ldots , x_n\}$,
\vspace*{0.2cm}\\  	
and $\displaystyle{ C_{\gamma} =  } \bigg\{C^{\prime}_{\gamma} |    C^{\prime}_{\gamma}~\mbox{is a minimal vertex cover of the induced path}~D(x_1, \ldots , x_{k-4})\bigg\} \\ 
\hspace*{0.75cm}\bigcup \{x_{k-3},x_{k-1}, x_{k}, x_{k+2}, x_{k+3}, \ldots , x_n\}$.
\vspace*{0.2cm}\\
By \cite[Theorem 3.6]{kanoy}, let  $\{C_{\delta}\}, \{C_{\zeta}\}$ and $\{C_{\nu}\}  $ be the collections of all maximal strong vertex covers of $ D_2 $, where

$\displaystyle{ C_{\delta} =  } \bigg\{C^{\prime}_{\delta} |    C^{\prime}_{\delta}~\mbox{is a minimal vertex cover of the induced path}~D(y_1, \ldots , y_{l-2})\bigg\} \\ 
\hspace*{0.75cm}\bigcup \{y_{l-1}, y_{l+1}, y_{l+2}, y_{l+3}, \ldots , y_m\}$,
\vspace*{0.2cm}\\
$\displaystyle{ C_{\zeta} =  } \bigg\{C^{\prime}_{\zeta} |    C^{\prime}_{\zeta}~\mbox{is a minimal vertex cover of the induced path}~D(y_1, \ldots , y_{l-3})\bigg\} \\ 
\hspace*{0.75cm}\bigcup \{y_{l-2},y_l, y_{l+1}, y_{l+2},  \ldots , y_m\}$,  	
\vspace*{0.2cm}\\
and $\displaystyle{ C_{\nu} =  } \bigg\{C^{\prime}_{\nu} |    C^{\prime}_{\nu}~\mbox{is a minimal vertex cover of the induced path}~D(y_1, \ldots , y_{l-4})\bigg\} \\ 
\hspace*{0.75cm}\bigcup \{y_{l-3},y_{l-1}, y_{l}, y_{l+2}, y_{l+3}, \ldots , y_m\}$.
\vspace*{0.2cm}\\
Let $C$ be a maximal strong vertex cover of $D$. Then by Lemma \ref{pathD1}, $ C \cap V(D_1) $  is a  maximal strong vertex cover of $ D_1 $ and similarly, $ C \cap V(D_2) $  is a  maximal strong vertex cover of $ D_2 $.   From the proof of Lemma \ref{pathD1}, we know $z_1 \in C$.  Therefore we can say that  $\{C_{i,j}\}$  be the collection of all the  maximal strong vertex covers of $ D $, where  $i \in \{\delta,\zeta, \nu  \}$,  $j \in \{\alpha,\beta, \gamma  \}$ and each  $C_{i,j} = C_{i} \cup \{z_1\}  \cup  C_{j}$.  By Theorem \ref{cooper}, we get  $$ \displaystyle{I^{(s)}  = \bigcap_{\substack{\ i \in \{\delta, \zeta, \nu\}\\~\mbox{and}~ j \in \{\alpha, \beta, \gamma\}}} (I_{\subseteq C_{i,j}})^s }.$$

By Theorem \ref{minimalgenerator}, we know that

$ I_{\subseteq C_{\delta,j}} = I_{\subseteq C_{\delta}} + (y_mz_1,x_nz_1) + I_{\subseteq C_{j}}  =  (C^{\prime}_{\delta}  ) + J_1^{\prime} + (y_mz_1,x_nz_1) + J_k + (C^{\prime}_{j}  )$ for $j= \alpha, \beta, \gamma$ when $k=1,2,3$, respectively,

$ I_{\subseteq C_{\zeta,j}} = I_{\subseteq C_{\zeta}} + (y_mz_1,x_nz_1) + I_{\subseteq C_{j}}  =  (C^{\prime}_{\zeta}  ) + J_2^{\prime} + (y_mz_1,x_nz_1) + J_k + (C^{\prime}_{j}  )$ for $j= \alpha, \beta, \gamma$ when $k=1,2,3$, respectively,

$ I_{\subseteq C_{\nu,j}} = I_{\subseteq C_{\nu}} + (y_mz_1,x_nz_1) + I_{\subseteq C_{j}}  =  (C^{\prime}_{\nu}  ) + J_3^{\prime} + (y_mz_1,x_nz_1) + J_k + (C^{\prime}_{j}  )$ for $j= \alpha, \beta, \gamma$ when $k=1,2,3$, respectively,

where $J_1^{\prime} = (y_{l-1}, y_{l+1}^{w_{l+1}},y_{l+1}y_{l+2}^{w_{l+2}},\ldots,y_{m-1}y_m^{w_m})$, $J_1 = (x_{k-1}, x_{k+1}^{w_{k+1}},x_{k+1}x_{k+2}^{w_{k+2}},\ldots,x_{n-1}x_n^{w_n})$,

$J_2^{\prime} = (y_{l-2},y_{l}^{w_{l}}, y_{l}y_{l+1}^{w_{l+1}},y_{l+1}y_{l+2}^{w_{l+2}},\ldots,y_{m-1}y_m^{w_m})$, $J_2 = (x_{k-2},x_{k}^{w_{k}}, x_{k}x_{k+1}^{w_{k+1}},x_{k+1}x_{k+2}^{w_{k+2}},\ldots,\\x_{n-1}x_n^{w_n})$,
$J_3^{\prime} = (y_{l-3},y_{l-1}, y_{l},y_{l+2}^{w_{l+2}},y_{l+2}y_{l+3}^{w_{l+3}},\ldots,y_{m-1}y_m^{w_m})$, $J_3 = (x_{k-3},x_{k-1}, x_{k},x_{k+2}^{w_{k+2}},\\x_{k+2}x_{k+3}^{w_{k+3}},\ldots,x_{n-1}x_n^{w_n})$.

%
%

The  rest of the proof follows with
similar argument as in  \cite[Theorem 3.6]{kanoy}. 
\end{proof}

\begin{lemma}\label{path5r}
	Let	$ D $ be the weighted oriented path  same as  defined in Notation \ref{notation1}.  Assume that    $w(y_m) \geq 2$ and $w(x_n) \geq 2$.  
	If  $w(y_l) \geq 2$ for some $1 < l \leq m-2$ such that  $w(y_{l+1}) =1$ or    $w(x_k) \geq 2$ for some $1 < k \leq n-2$ such that  $w(x_{k+1}) = 1$, then $ I(D)^{(s)} \neq I(D)^{s} $ for some $s \geq 2$.  
	
\end{lemma}

\begin{proof}
Suppose  $w(y_l) \geq 2$ for some $1 < l \leq m-2$ such that  $w(y_{l+1}) =1$. 

\textbf{Case (1)}  Assume that  $1 < l \leq m-3$.

  Let $ D^{\prime} $ be a weighted naturally oriented path with
	$ V(D^{\prime}) = \{y_1, y_2,  \ldots , y_l,y_{l+1},y_{l+2},y_{l+3}\}$  and $E(D^{\prime}) = \{(y_i,y_{i+1}) ~|~ 1 \leq i \leq l+2 \}$.	
	Then by Theorem \ref{kanoy2},   $ {I(D^{\prime})}^{(s)} \neq {I(D^{\prime})}^s $ for some $s \geq 2$ and so  
	by Theorem \ref{inducedpath}, we have  $ I(D)^{(s)} \neq I(D)^{s} $.
	
\textbf{Case (2)}  Assume that  $ l = m-2$.

Here $w(y_{m-1}) = 1$. Let $ D^{\prime} $ be a weighted naturally oriented path with
$ V(D^{\prime}) = \{y_1, y_2,  \ldots ,\\ y_{m-2},y_{m-1},y_{m},z_1\}$  and $E(D^{\prime}) = \{(y_i,y_{i+1}) ~|~ 1 \leq i \leq m-1 \} \cup \{(y_m,z_1)\}$.	
Then by Theorem \ref{kanoy2},   $ {I(D^{\prime})}^{(s)} \neq {I(D^{\prime})}^s $ for some $s \geq 2$ and thus  
by Theorem \ref{inducedpath}, we get  $ I(D)^{(s)} \neq I(D)^{s} $.

Similarly	if $w(x_k) \geq 2$ for some $1 < k \leq n-2$ such that  $w(x_{k+1}) = 1$,  we can show that $ I(D)^{(s)} \neq I(D)^{s} $ for some $s \geq 2$. 	
\end{proof}

\begin{theorem}\label{path6}
	Let	$ D $ be the weighted oriented path  same as defined in Notation \ref{notation1}. Assume that   $w(y_m) \geq 2$ and $w(x_n) \geq 2$.
 Then $ I(D)^{(s)} = I(D)^{s} $ for all $s \geq 2$  if and only if  $D$
	satisfies the condition  ``there exist two indices $ 1 < l < m $ and $ 1 < k < n $ such that the set of vertices with non-trivial
	weights are precisely $ \{y_l
	,\ldots, y_m\} \cup \{x_k,\ldots, x_n\}  $".  
	
\end{theorem}

\begin{proof}
It follows from Theorem \ref{path5}  and  Lemma \ref{path5r}.	
\end{proof}

\begin{remark}
Let	$ D $ be the weighted oriented path  same as defined in Notation \ref{notation1}. If  $w(x_i) = 1$ for  $2 \leq i \leq n$, then by changing the orientations of  edges of the edge set  $ \{(x_i,x_{i+1}) ~|~ 1 \leq i \leq n-1 \}\cup \{(x_n,z_1)\}$, we can think $D$ as a weighted naturally oriented path. Similarly if  $w(y_i) = 1$ for  $2 \leq i \leq m$,  we can think $D$ as a weighted naturally oriented path.

\end{remark}

\begin{notation}\label{notation2}
	Let	$ D $ be a weighted oriented  path with
	$ V(D) = \{y_1, y_2, \ldots , y_m, z_1, z_2, x_n, \\  \ldots ,  x_2, x_1\}$, $E(D) = \{(y_i,y_{i+1}) ~|~ 1 \leq i \leq m-1 \} \cup \{(y_m,z_{1}), (z_1,z_2), (x_n,z_{2})\} \cup   \{( x_i,  x_{i+1}) ~|~\\1 \leq i \leq n-1 \}$, $ w(z_1)=1$ and $ w(z_2)=1$ (see Figure \ref{fig.8}).

\begin{figure}[!ht]
	\begin{tikzpicture}[scale=0.75]
		\begin{scope}[ thick, every node/.style={sloped,allow upside down}] 
			\definecolor{ultramarine}{rgb}{0.07, 0.04, 0.56} 
			\definecolor{zaffre}{rgb}{0.0, 0.08, 0.66}
			\draw[fill, black] (-6,0) --node {\midarrow}(-4.5,0);
			\draw[fill, black][dotted,thick] (-3,0)   -- (-4.5,0);
			\draw[fill, black]  (-3,0)  --node {\midarrow}(-1.5,0);
			\draw[fill, black][dotted,thick] (0,0) --(-1.5,0);   
			\draw[fill, black]  (0,0)--node {\midarrow}(1.5,0);
			\draw[fill, black] (1.5,0) --node {\midarrow}(3,0);
			\draw[fill, black] (3,0) --node {\midarrow}(4.5,0);
			\draw[fill, black] (6,0) --node {\midarrow}(4.5,0);
			\draw[fill, black] (7.5,0) --node {\midarrow} (6,0);
			\draw [fill, black][dotted,thick](9,0) --(7.5,0);
			\draw[fill, black]  (10.5,0) --node {\midarrow}(9,0);
			\draw[fill, black] [dotted,thick] (12,0) --(10.5,0);
			\draw[fill, black]  (13.5,0) --node {\midarrow}(12,0);

			\draw[fill] [fill] (-1.5,0) circle [radius=0.04];
			\draw[fill] [fill] (-3,0) circle [radius=0.04];
			\draw[fill] [fill] (-4.5,0) circle [radius=0.04];
			\draw[fill] [fill] (-6,0) circle [radius=0.04];  
			\draw [fill] [fill] (0,0) circle [radius=0.04];
			\draw[fill] [fill] (1.5,0) circle [radius=0.04];
			\draw[fill] [fill] (3,0) circle [radius=0.04];
			\draw[fill] [fill] (4.5,0) circle [radius=0.04];
			\draw[fill] [fill] (6,0) circle [radius=0.04];
			\draw[fill] [fill] (7.5,0) circle [radius=0.04];
			\draw[fill] [fill] (9,0) circle [radius=0.04];
			\draw[fill] [fill] (10.5,0) circle [radius=0.04];
			\draw[fill] [fill] (12,0) circle [radius=0.04];
			\draw[fill] [fill] (13.5,0) circle [radius=0.04];

			\node at (-6,0.5) {$y_1$};
			\node at (-4.5,0.5) {$y_{2}$};
			\node at (-3,0.5) {$y_{l}$};
			\node at (-1.5,0.5) {$y_{l+1}$};
			\node at (0,0.5) {$y_{m-1}$};
			\node at (1.5,0.5) {$y_m$};
			\node at (3,0.5) {$z_1$};
			\node at (4.5,0.5) {$z_{2}$};
			\node at (6,0.5) {$x_{n}$};
			\node at (7.5,0.5) {$x_{n-1}$};
			\node at (9,0.5) {$x_{k+1}$};  
			\node at (10.5,0.5) {$x_{k}$};
			\node at (12,0.5) {$x_{2}$};
			\node at (13.5,0.5) {$x_{1}$};   
		\end{scope}
	\end{tikzpicture}
	\caption{An oriented path which is union of two naturally oriented paths joined at a sink vertex.}\label{fig.8}
\end{figure}
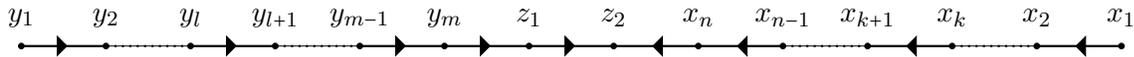

\end{notation}

\begin{theorem}\label{path7.}
Let	$ D $ be the weighted oriented path  same as defined in Notation \ref{notation2}. Assume that
there exist two indices $ 1 < l \leq m $ and $ 1 < k \leq n $ such that the set of vertices with non-trivial
weights are precisely $ \{y_l
,\ldots, y_m\} \cup \{x_k,\ldots, x_n\}. $
 Then $ I(D)^{(s)} = I(D)^{s} $ for all $s \geq 2$.  
	
\end{theorem}

\begin{proof}

This proof follows by the similar argument as in Theorem \ref{path5}.	
\end{proof}

\begin{lemma}\label{path7}
	Let	$ D $ be the weighted oriented path  same as defined in Notation \ref{notation2}.  
	If  $w(y_l) \geq 2$ for some $1 < l \leq m-1$ such that  $w(y_{l+1}) =1$  or      $w(x_k) \geq 2$ for some $1 < k \leq n-1$ such that  $w(x_{k+1}) = 1$, then $ I(D)^{(s)} \neq I(D)^{s} $ for some $s \geq 2$.  
	
\end{lemma}

\begin{proof}
Suppose $w(y_l) \geq 2$ for some $1 < l \leq m-1$ such that  $w(y_{l+1}) =1$.

\textbf{Case (1)}  Assume that  $1 < l \leq m-2$.

By the same argument as in Lemma \ref{path5r},  we can show that  $ I(D)^{(s)} \neq I(D)^{s} $  for some $s \geq 2$.

\textbf{Case (2)}  Assume that  $ l = m-1$.

Here $w(y_m) = 1$. Let $ D^{\prime} $ be a weighted naturally oriented path with
$ V(D^{\prime}) = \{y_1, y_2,  \ldots ,\\ y_{m-1},y_{m},z_1,z_2\}$  and $E(D^{\prime}) = \{(y_i,y_{i+1}) ~|~ 1 \leq i \leq m-1 \} \cup \{(y_m,z_1), (z_1,z_2)\}$.	
Then by Theorem \ref{kanoy2},   $ {I(D^{\prime})}^{(s)} \neq {I(D^{\prime})}^s $ for some $s \geq 2$ and thus  
by Theorem \ref{inducedpath}, we get  $ I(D)^{(s)} \neq I(D)^{s} $.
\vspace*{0.1cm}\\
Suppose  $w(x_k) \geq 2$ for some $1 < k \leq n-1$ such that  $w(x_{k+1}) = 1$.  Since $w(z_1)=w(z_2)=1$, we can assume $(z_2,z_1) \in E(D)$.  Let $ D^{\prime\prime} $ be a weighted naturally oriented path with
$ V(D^{\prime\prime}) = \{x_1, x_2,  \ldots , x_{n-1},x_{n},z_2,z_1\}$  and $E(D^{\prime\prime}) = \{(x_i,x_{i+1}) ~|~ 1 \leq i \leq n-1 \}\cup \{(x_n,z_{2}), (z_2,z_1)\}$. By  the similar argument as for $D^{\prime}  $, we can show that $ I(D)^{(s)} \neq I(D)^{s} $, for some $s \geq 2$. 	
\end{proof}

\begin{theorem}\label{path7r}
	Let	$ D $ be the weighted oriented path  same as defined in Notation \ref{notation2}. Assume that  $w(x_i) \geq 2$ and $w(y_j) \geq 2$ for some $i$ and $j$. Then $ I(D)^{(s)} = I(D)^{s} $ for all $s \geq 2$  if and only if  $D$
	satisfies the condition ``there exist two indices $ 1 < l \leq m $ and $ 1 < k \leq n $ such that the set of vertices with non-trivial
	weights are precisely $ \{y_l
	,\ldots, y_m\} \cup \{x_k,\ldots, x_n\}  $".  
	
\end{theorem}

\begin{proof}
	It follows from Theorem \ref{path7.}  and  Lemma \ref{path7}.	
\end{proof}

\begin{theorem}\label{path6.}
	Let	$ D $ be the weighted oriented path  same as defined in Notation \ref{notation1}. Assume that   $w(x_i) \geq 2$ and $w(y_j) \geq 2$ for some $i$ and $j$.  Then $ I(D)^{(s)} = I(D)^{s} $ for all $s \geq 2$	if and only if

\begin{enumerate}

\item  When $w(x_n) = 1$, 
$D$
satisfies the condition ``there exist two indices $ 1 < l \leq m $ and $ 1 < k < n $ such that the set of vertices with non-trivial
weights are precisely $ \{y_l
,\ldots, y_m\} \cup \{x_k,\ldots, x_{n-1}\}  $".

\item When $w(y_m) = 1$,   
$D$  satisfies the condition ``there exist two indices $ 1 < l < m $ and $ 1 < k \leq n $ such that the set of vertices with non-trivial
weights are precisely $ \{y_l
,\ldots, y_{m-1}\} \cup \{x_k,\ldots, x_n\}  $".

\item  When $w(y_m) \neq 1$ and $w(x_n) \neq 1$,  
$D$
satisfies the condition ``there exist two indices $ 1 < l < m $ and $ 1 < k < n $ such that the set of vertices with non-trivial
weights are precisely $ \{y_l
,\ldots, y_m\} \cup \{x_k,\ldots, x_n\}  $".  
\end{enumerate}

\end{theorem}

\begin{proof}
(1)  Since $w(x_n) = 1$, we can assume $(z_1,x_n) \in E(D)$.  If  we rename the vertex   $x_n$ by $z_2$, then the proof follows from  Theorem \ref{path7r}.

(2)  Note that  $w(y_m) = 1$.  If  we rename the vertices $y_m$ and $ z_1  $ by $z_1$ and $z_2$, respectively, then the proof follows from  Theorem \ref{path7r}.

(3) It follows from  Theorem \ref{path6}.
\end{proof}

\begin{theorem}\label{path6r}
	Let	$ D $ be a weighted oriented path with
	$ V(D) = \{y_1, y_2, \ldots , y_m, z_1,  x_n,   \ldots ,\\  x_2, x_1\}$  and $E(D) = \{(y_i,y_{i+1}) ~|~ 1 \leq i \leq m-1 \} \cup \{(y_m,z_{1}),  (x_n,z_{1})\} \cup   \{( x_i,  x_{i+1}) ~|~ 1 \leq i \leq n-1 \}$. Assume that   $w(x_i) \geq 2$ and $w(y_j) \geq 2$ for some $i$ and $j$.  Then $ I(D)^{(s)} = I(D)^{s} $ for all $s \geq 2$	if and only if

	\begin{enumerate}

		\item  When $w(x_n) = 1$, 
		$D$
		satisfies the condition ``there exist two indices $ 1 < l \leq m $ and $ 1 < k < n $ such that the set of vertices except $z_1$ with non-trivial
		weights  are precisely $ \{y_l
		,\ldots, y_m\} \cup \{x_k,\ldots, x_{n-1}\}  $".

		\item When $w(y_m) = 1$,   
		$D$  satisfies the condition ``there exist two indices $ 1 < l < m $ and $ 1 < k \leq n $ such that the set of vertices except $z_1$ with non-trivial
		weights  are precisely $ \{y_l
		,\ldots, y_{m-1}\} \cup \{x_k,\ldots, x_n\}  $".

		\item  When $w(y_m) \neq 1$ and $w(x_n) \neq 1$,  
		$D$
		satisfies the condition ``there exist two indices $ 1 < l < m $ and $ 1 < k < n $ such that the set of vertices except $z_1$ with non-trivial
		weights  are precisely $ \{y_l
		,\ldots, y_m\} \cup \{x_k,\ldots, x_n\}  $".  
	\end{enumerate}

\end{theorem}

\begin{proof}
Here $z_1$ is the only sink vertex in $D$. If $w(z_1) = 1$,
the proof follows from    Theorem \ref{path6.} and 
if $w(z_1) \geq 2 $, then  the proof follows from  Lemma \ref{cor3}  and   Theorem \ref{path6.}.
\end{proof}

\section{Symbolic powers of weighted rooted trees}

In the computation of symbolic powers of edge ideals of weighted oriented graphs,  we always need to know all the maximal strong vertex covers. In this section, we  give a new technique to find all the maximal strong vertex covers of a particular class of weighted  oriented  graphs. Finally, we show the  equality of ordinary and symbolic powers of edge ideals of certain class of weighted rooted trees.

\begin{lemma}\cite[Lemma 47]{pitones}\label{v.c}
	Let $ D $ be  a  weighted  oriented  graph  such  that $  I(D) \subseteq (x_{i_1}^{a_1},\ldots,x_{i_s}^{a_s}).$ Then $ \{{x_{i_1}},\ldots,{x_{i_s}}\} $ is a vertex cover of $ D $.   	
\end{lemma}

After  fixing the value of each $a_i$ with its corresponding weight, we get strong vertex cover in the following  result.

\begin{lemma}\label{maximal}
	Let $ D $ be  a  weighted  oriented  graph  such  that $  I(D) \subseteq (x_{i_1}^{w_{i_1}},\ldots,x_{i_s}^{w_{i_s}})$, where $s$ is the least possible value. Then $ \{{x_{i_1}},\ldots,{x_{i_s}}\} $ is a maximal strong vertex cover of $ D $.     	
\end{lemma}

\begin{proof}  
	Let $C = \{{x_{i_1}},\ldots,{x_{i_s}}\}.$ By Lemma \ref{v.c}, $C$ is a  vertex cover of $ D $. Let $x_{i_j} \in  L_3^D(C).$ By Lemma \ref{L3}, $  N_D(x_{i_j}) \subset C.$ If $  N_D^{-}(x_{i_j}) =\phi,$ then  $  I(D) \subseteq (x_{i_1}^{w_{i_1}},\ldots,x_{i_{j-1}}^{w_{i_{j-1}}}, x_{i_{j+1}}^{w_{i_{j+1}}} ,\ldots,\\x_{i_s}^{w_{i_s}}).$ Since $s$ is the least possible value, its a contradiction. Therefore $  N_D^{-}(x_{i_j})    \neq \phi.$ Let  $  N_D^{-}(x_{i_j}) =\{{x_{k_1}},\ldots,{x_{k_t}}    \}.$ If $w_{k_1}=\cdots={w_{k_t}}=1 ,$ then $  I(D) \subseteq (x_{i_1}^{w_{i_1}},\ldots,x_{i_{j-1}}^{w_{i_{j-1}}}, x_{i_{j+1}}^{w_{i_{j+1}}} ,\ldots,\\x_{i_s}^{w_{i_s}})$ and its a contradiction  by the previous argument. Thus at least one of $x_{k_1},\ldots,{x_{k_t}}$ has non-trivial weight. Without loss of generality let $w_{k_1}  \neq 1  $. Suppose $x_{k_1} \in L_1^D(C)$. That means  there is some  $x_{l_1} \in N_D^{+}(x_{k_1}) \cap C^c$. Here $({x_{k_1},x_{l_1}}) \in E(D)$. Since $x_{k_1}  \notin (x_{i_1}^{w_{i_1}},\ldots,x_{i_s}^{w_{i_s}})    $ and   $x_{l_1} \notin C$, we have   $x_{k_1}x_{l_1}^{w_{l_1}}  \notin (x_{i_1}^{w_{i_1}},\ldots,x_{i_s}^{w_{i_s}})    $. Thus $  I(D) \nsubseteq (x_{i_1}^{w_{i_1}},\ldots,\\x_{i_s}^{w_{i_s}})$, which is a contradiction.
 Therefore  $x_{k_1} \in L_2^D(C) \cup L_3^D(C)$.  Note that  $ x_{k_1}  \in    N_D^{-}(x_{i_j})\cap V^+(D) \cap [L_2^D(C) \cup L_3^D(C)] $, i.e.,    $ x_{i_j} $ satisfies SVC condition
	on  $ C $. Hence $C$ is  strong.  Suppose $C$ is not maximal. That means,  there exists a strong vertex cover  $ C^{\prime}$   of $D$  such that  
	$C \subsetneq C^{\prime}$. Let $x_q \in  C^{\prime}  \setminus  C.$ Since $x_q \notin C$  and $C$ is a vertex cover of $D$, $N_D(x_q)  \subset C  $. Here   $N_D(x_q)  \subset  C^{\prime}$ and so by Lemma \ref{L3}, $x_q \in L_3^D(C^{\prime})$. Since $C^{\prime}$ is strong, there is some $x_{p}  \in N_D^{-}(x_{q}) \cap V^+(D) \cap [{L_2^D{(C^{\prime})}} \cup   {L_3^D{(C^{\prime})}} ]$. Note that $x_p \in C$,  $w(x_p) \geq 2$, $x_q \notin C$ and $(x_p,x_q) \in E(D)$. Then  $  x_px_q^{w_q} \notin (x_{i_1}^{w_{i_1}},\ldots,x_{i_s}^{w_{i_s}})$ and  so  $  I(D) \nsubseteq (x_{i_1}^{w_{i_1}},\ldots,x_{i_s}^{w_{i_s}})$, which is a contradiction. Hence $C$ is maximal.     	
\end{proof}

\begin{remark}
Converse of the above lemma need not be true in general.

\hspace*{0.5cm}For example consider the  weighted oriented path $D$ as in Figure \ref{fig.6}.  Then $I(D) =   (x_1x_2^{7},x_2x_3,x_3x_4)$. Using Macaulay $2$,  the strong vertex covers of $ D $ are $ \{x_1, x_3\},$ $ \{x_2, x_3\}$  and    $\{x_2,x_4\}$.  Note that $\{x_2,x_4\}$ is a maximal strong vertex cover, but $I(D) \nsubseteq    (x_2^{7},x_4)$  because    $x_2x_3 \notin  (x_2^{7},x_4)$.	
\begin{figure}[!ht]
	\begin{tikzpicture}[scale=0.7]
	\begin{scope}[ thick, every node/.style={sloped,allow upside down}] 
	\definecolor{ultramarine}{rgb}{0.07, 0.04, 0.56} 
	\definecolor{zaffre}{rgb}{0.0, 0.08, 0.66}
	
	\draw [fill, black](0,0) --node {\midarrow}(1.5,0);  
	\draw [fill, black](1.5,0) --node {\midarrow}(3,0);
	\draw [fill, black](3,0) --node {\midarrow}(4.5,0);

	\draw[fill] [fill] (0,0) circle [radius=0.04];
	\draw[fill] [fill] (1.5,0) circle [radius=0.04];
	\draw[fill] [fill] (3,0) circle [radius=0.04];
	\draw[fill] [fill] (4.5,0) circle [radius=0.04];

	\node at (0,0.5) {$x_1$};
	\node at (1.5,0.5) {$x_2$};
	\node at (3,0.5) {$x_{3}$};  
	\node at (4.5,0.5) {$x_{4}$};
	
	\node at (0,-0.5) {$1$};
	\node at (1.5,-0.5) {$7$};
	\node at (3,-0.5) {$1$};  
	\node at (4.5,-0.5) {$1$};

	\node at (2.5,-1.2) {$D$};


	\end{scope}
	\end{tikzpicture}
	\caption{A weighted oriented path $D$ of length $3$.}\label{fig.6}
\end{figure}
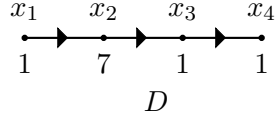

\end{remark}

We see that the converse of Lemma \ref{maximal} is true under certain condition on  weights of vertices and orientation of edges of $D$.  To prove the converse part, the following lemma is  important.

\begin{lemma}\label{maximal3}
	Let $ D $ be a  weighted oriented graph such that at most one edge is oriented into each vertex  and $w(x) \geq 2$ if $\deg_D(x)\geq 2 $ for all $x \in V(D)$.  Let $C$ be a strong  vertex cover of $D$ with $x_i \in  L_1^{D}(C) $ and $w(x_i)  \neq 1$. Then there exists a strong vertex cover  $ C^{\prime}$   of $D$  such that  
	$C \subsetneq C^{\prime}$.
\end{lemma}

\begin{proof}
	
		Let $L =\{  x~|~x \in L_1^D(C) \cap V^+(D) \}=   \{x_{i_1},x_{i_2},\ldots,x_{i_l}\}  $.  By our  assumption  $L  \neq  \phi$.  Let $ J = [N_{D}^+ (x_{i_1})\cap {{C}}^c] \cup \cdots \cup  [N_{D}^+ (x_{i_l})\cap {{C}}^c]  = \{x_{j_1},x_{j_2},\ldots,x_{j_r}\}$. Let $C^{\prime} = C  \cup J$. Then    $ x_{i_m} \notin    L_1^{D}(C^{\prime}) $  because $ N_{D}^+ (x_{i_m})\cap {{C^{\prime}}}^c = \phi $ for $1 \leq m \leq l$.  Since  $J \subset C^c$ and  $C$ is a vertex cover of $D$, $N_D(x_{j_n}) \subseteq C \subset {{C^{\prime}}}$  for each $x_{j_n}  \in J$. By Lemma \ref{L3}, $J \subset L_3^{D}(C^{\prime}) $. We claim $C^{\prime}$ is a strong vertex cover of $D$.  
	   Note that $     L_3^{D} ({{{C}}}) \cup \{x_{j_1},x_{j_2},\ldots,x_{j_r}\}   \subseteq  L_3^{D}(C^{\prime})$. Let $x_l \in L_3^{D}(C)$. Then $ x_l $ satisfies SVC condition
	   on  $ {C} $ because $C$ is strong. By Lemma \ref{svc1},  $x_l \in L_3^{D}(C^{\prime})$  and  $ x_l $ satisfies SVC condition
	   on  $ C^{\prime} $.  Without loss of generality let $ x_{j_1} \in   N_{D}^+ (x_{i_1}) $.  
	    Since $  x_{i_1} \in N_{D}^- (x_{j_1}) \cap V^+(D)\cap  [L_2^{D}(C^{\prime}) \cup  L_3^{D}(C^{\prime})]$,  $x_{j_1}$
	satisfies SVC condition on $ C^{\prime} $. By the similar argument, we  can show $x_{j_t}$
	satisfies SVC condition on $ C^{\prime} $ for $2 \leq t \leq r$.  	If $     L_3^{D} ({{{C}}}) \cup \{x_{j_1},x_{j_2},\ldots,x_{j_r}\}   =  L_3^{D}(C^{\prime})$, then  each element of $L_3^{D}(C^{\prime})$  satisfies SVC condition on $ C^{\prime} $ and hence $ C^{\prime} $ is  strong.  Now we assume that  $     L_3^{D} ({{{C}}}) \cup \{x_{j_1},x_{j_2},\ldots,x_{j_r}\}   \subsetneq  L_3^{D}(C^{\prime})$.
Let  $x_k  \in   L_3^{D}(C^{\prime}) \setminus [L_3^{D} ({{{C}}}) \cup \{x_{j_1},x_{j_2},\ldots,x_{j_r}\} ]$. Then $ x_k $ lies in the neighbourhood of  $ x_{j_t} $ for
	some $ t \in  [r]$.  Without loss of generality we can assume that $ x_{j_t} \in   N_{D}^+ (x_{i_1}) $. By definition of $D$,  $  N_{D}^- (x_{j_t})  =  \{x_{i_1}\}$.

\textbf{Case (1)} Suppose  $x_k \in N_{D}^- (x_{j_t})$.  Then  $x_k = x_{i_1} $.

 Since $w(x_{i_1})  \geq  2  ,$ there is $ (x_p,x_{i_1})  \in  E(D) $ for some $x_p \in  V(D)$.   If $\deg_D(x_p) = 1$, then $N_D(x_p) = \{x_{i_1}\}$.  Since  $x_{i_1}  \in   L_3^{D}(C^{\prime})$, by Lemma \ref{L3}, $N_D(x_{i_1}) \subset  C^{\prime}$  and  so  $ x_p \in  C^{\prime}   $. Here $x_p \notin J$ because  $ N_{D}^- (x_{p})  =  \phi$. This implies $ x_p \in  C$.  We know $ x_{i_1} \in C$. Then by Lemma \ref{L3},   $x_p  \in L_3^D(C)$. Since $ N_{D}^- (x_{p})  =  \phi$,  $x_p  $  does not satisfy the SVC condition on $C$. Its  a contradiction because $C$  is strong. Therefore  $\deg_D(x_p) \geq 2$ and by the definition of $D$, $w(x_p) \geq 2$. Suppose  $x_p  \in  L_1^{D}(C^{\prime})$.  That means there exists some     $x_q \in N_{D}^+ (x_p)\cap {C^{\prime}}^c $.  Since $C \subsetneq  C^{\prime}$,  $x_q \in N_{D}^+ (x_p)\cap {C}^c $. Thus  $x_p  \in  L_1^{D}(C) $  and  so $x_p \in L$, which is a contradiction because each element of $L$ $  \notin  L_1^{D}(C^{\prime})$. Therefore $x_p  \notin  L_1^{D}(C^{\prime})$.    Then   $x_{p} \in N_D^{-}(x_{i_1}) \cap V^+(D) \cap [{L_2^D{(C^{\prime})}} \cup   {L_3^D{(C^{\prime})}} ]$, i.e., $ x_{i_1} = x_k $ satisfies SVC condition
on  $ C^{\prime} $.

\textbf{Case (2)} Suppose $x_k \in N_{D}^+ (x_{j_t})$.  Then $x_k \neq x_{i_1} $.               

 Note that  $(x_{i_1},x_{j_t})$ and $(x_{j_t},x_k)  \in  E(D)$.
 Since $\deg_D(x_{j_t}) \geq  2$, by the definition of $D$,  $w(x_{j_t}) \geq  2$.  This implies   $x_{j_t} \in N_D^{-}(x_k) \cap V^+(D) \cap   {L_3^D{(C^{\prime})}} $, i.e., $ x_k $ satisfies SVC condition
	on  $ C^{\prime} $.

Therefore each element of $L_3^{D}(C^{\prime})$  satisfies SVC condition on $ C^{\prime} $. Hence $ C^{\prime} $ is  strong.
\end{proof}

\begin{theorem}\label{maximal2}  
	Let $ D $ be a  weighted oriented graph such that at most one edge is oriented into each vertex  and $w(x) \geq 2$ if $\deg_D(x)\geq 2 $ for all $x \in V(D)$. Then $  I(D) \subseteq (x_{i_1}^{w_{i_1}},\ldots,x_{i_s}^{w_{i_s}})$, where $s$ is the least possible value if and only if $ \{{x_{i_1}},\ldots,{x_{i_s}}\} $ is a maximal strong vertex cover of $ D $.            	
\end{theorem}

\begin{proof}
If  $  I(D) \subseteq (x_{i_1}^{w_{i_1}},\ldots,x_{i_s}^{w_{i_s}})$, where $s$ is the least possible value, then by Lemma \ref{maximal}, $ \{{x_{i_1}},\ldots,{x_{i_s}}\} $ is a maximal strong vertex cover of $ D $.

Now we assume that  $C= \{{x_{i_1}},\ldots,{x_{i_s}}\} $ is a maximal strong vertex cover of $ D $.
We claim  $  I(D) \subseteq (x_{i_1}^{w_{i_1}},\ldots,x_{i_s}^{w_{i_s}})$, where $s$ is the least possible value. Let $x_ix_j^{w_j}  \in I(D)$. If $x_j \in C,$ then  $x_ix_j^{w_j}  \in   (x_{i_1}^{w_{i_1}},\ldots,x_{i_s}^{w_{i_s}}) $. Suppose $x_j \notin C$.  Since $(x_i,x_j) \in E(D)$, $  x_j  \in  N_D^+(x_i)  \cap C^c$ and it implies $x_i  \in  L_1^D(C)$.  If $w(x_i) \neq 1$,
  by  Lemma \ref{maximal3},  there exists a strong vertex cover  $ C^{\prime}$   of $D$  such that  
$C \subsetneq C^{\prime}$, which is a contradiction because $C$  is maximal. So $w(x_i) = 1$ and  $x_ix_j^{w_j}  \in   (x_{i_1}^{w_{i_1}},\ldots,x_{i_s}^{w_{i_s}}) $. Therefore
 $  I(D) \subseteq (x_{i_1}^{w_{i_1}},\ldots,x_{i_s}^{w_{i_s}})$. Suppose  $s$ is not the least possible value.  That means,   $  I(D) \subseteq (x_{j_1}^{w_{j_1}},\ldots,x_{j_r}^{w_{j_r}})$ for some $\{x_{j_1},\ldots,  x_{j_r} \} \subsetneq  \{x_{i_1},\ldots,  x_{i_s} \}   $, where $r  <  s$ and $r$ is the least possible value.  Then by Lemma \ref{maximal}, $\{x_{j_1},\ldots,  x_{j_r} \} $ is a maximal strong vertex cover of $ D $, which is a contradiction because $C$  is maxmal.        
\end{proof}

Next we prove the  equality of ordinary and symbolic powers of edge ideals of some class of weighted rooted trees.

\begin{theorem}\label{tree}  
	Let $ D $ be a weighted  rooted tree with root $x_0$, $\deg_D(x_0) = 1$ and $w(x) \geq 2$ if $\deg_D(x)\geq 2 $ for all $x \in V(D)$.  Then $ {I(D)}^{(s)} = {I(D)}^s $ for all $s \geq   2$.
\end{theorem}
\begin{proof}
Let $N_D(x_0)  =  \{x_1\}$ and by definition of $D$, $(x_0,x_1) \in E(D)$. Here $N_D^-(x_1) = \{x_0\}$. Let $N_D^+(x_1) = \{x_2,x_3,\ldots,x_r\}$.   Note that $I(D)  \subseteq  (\{x_i^{w_i}~|~x_i \in V(D)  \setminus  \{x_0\}   \})    $ and  $I(D)  \nsubseteq  (\{x_i^{w_i}~|~x_i \in V(D)  \setminus  \{x_0,x_p\}   \})   $  for any $x_p \in V(D)  \setminus  \{x_0\} $. Thus by Theorem \ref{maximal2}, $V(D)  \setminus  \{x_0\}$ is a maximal strong vertex cover of $D$.  Also $I(D)  \subseteq  (\{x_i^{w_i}~|~x_i \in V(D)  \setminus  \{x_1\}   \})    $ and  $I(D)  \nsubseteq  (\{x_i^{w_i}~|~x_i \in V(D)  \setminus  \{x_1,x_q\}   \})   $  for any $x_q \in V(D)  \setminus  \{x_1\} $. Thus by Theorem \ref{maximal2}, $V(D)  \setminus  \{x_1\}$ is a maximal strong vertex cover  of $D$. Let $C_1= V(D)  \setminus  \{x_0\}$ and  $C_2= V(D)  \setminus  \{x_1\}$.
Suppose there exists a maximal strong vertex cover $C$ of $D$, which contains both $x_0$ and $x_1.$  Then we can write  $C= \{x_0,x_1,{x_{i_1}},\ldots,x_{i_{s-2}}\} $, where    $\{x_{i_1},\ldots,x_{i_{s-2}}  \}  \subseteq [V(D)\setminus \{x_0,x_1\}]$ is some vertex set.  By Theorem \ref{maximal2},  $  I(D) \subseteq (x_0,x_1^{w_1},x_{i_1}^{w_{i_1}},\ldots,x_{i_{s-2}}^{w_{i_{s-2}}})$, where    $|C|=s$ is the least possible value. Since $x_0x_1^{w_1}$ is the only minimal generator of $I(D)$,  which involves  the vertex $x_0$, we have $  I(D) \subseteq (x_1^{w_1},x_{i_1}^{w_{i_1}},\ldots,x_{i_{s-2}}^{w_{i_{s-2}}})$. That means $s$ is not  the least possible value and by Theorem \ref{maximal2}, its a contradiction.
Hence	$C$ is not a maximal strong vertex cover of $D$. Thus $C_1$ and $C_2$ are the only maximal strong vertex covers of $D.$ Let $D^{\prime} = D \setminus \{x_0\}  $ and $D^{\prime\prime} = D  \setminus  \{x_0,x_1\} $ be the induced digraphs of $D.$      
Here $ L_2^D(C_1) = \{x_1\}$, $L_3^D(C_1)  = V(D)  \setminus  \{x_0,x_1\} $ and  $w(x_i) \geq 2$ if $\deg_D(x)\geq 2 $  for all $x_i \in C_1$. Then by Theorem \ref{minimalgenerator}, $I_{\subseteq C_1}  = (x_1^{w_1}) + I(D^{\prime})    $.  Here $ L_1^D(C_2) = \{x_0\}$, $ L_2^D(C_2) = \{x_2,\ldots,x_r\}$, $L_3^D(C_2)  =  V(D)  \setminus  \{x_0,x_1,x_2,  \ldots, x_r\}  $  and  $w(x_i) \geq 2$ if $\deg_D(x)\geq 2 $  for all $x_i \in C_2$. Then by Theorem \ref{minimalgenerator}, $I_{\subseteq C_2}  = (x_0,x_2^{w_2},\ldots,x_r^{w_r}) + I(D^{\prime\prime})    $. Let $I(D^{\prime\prime})  = (f_1,\ldots,f_t)     $. Hence  $I_{\subseteq C_1}  = (x_1^{w_1}) + I(D^{\prime})  =  (x_1^{w_1},x_1x_2^{w_2},\ldots,  x_1x_r^{w_r}, f_1,\ldots,f_t)$ and $I_{\subseteq C_2}  =   (x_0,x_2^{w_2},\ldots,x_r^{w_r}, f_1,\ldots,f_t)$.

By Lemma \ref{cooper}, we have
$   I(D)^{(s)} 
  = (I_{\subseteq C_1})^s \cap (I_{\subseteq C_2})^s   $. It is enough to prove that
$   I(D)^{(s)} 
= (I_{\subseteq C_1})^s \cap (I_{\subseteq C_2})^s \subseteq    I(D)^{s}    $.  We prove this by induction on $s.$ The case for $s=1$ is trivial. Let $m \in \mathcal{G}(I(D)^{(s)}) $. Then $m = \lcm(m_1,m_2)$ for some $m_1 \in \mathcal{G}(( x_1^{w_1},x_1x_2^{w_2},\ldots,  x_1x_r^{w_r}, f_1,\ldots,f_t )$ and $m_2 \in \mathcal{G}((x_0,x_2^{w_2},\ldots,x_r^{w_r}, f_1,\ldots,f_t  )).$
Thus  $ m_1 =  (x_1^{w_1})^{a_1 } (x_1x_2^{w_2})^{a_2}\ldots(x_1x_r^{w_r})^{a_r}(f_1)^{a_{r+1}}\cdots(f_t)^{a_{r+t}}   $  and
$ m_2 =  (x_0)^{b_1 } (x_2^{w_2})^{b_2}\ldots(x_r^{w_r})^{b_r}(f_1)^{b_{r+1}}\\\cdots(f_t)^{b_{r+t}} $ for some $a_i, b_i \geq 0$ with  $\displaystyle{\sum_{i=1}^{r+t} }a_i= s$ and $ \displaystyle{\sum_{i=1}^{r+t} }b_i= s$.  	  
	
\textbf{Case (1)} Assume that $a_1 \neq 0.$

If $b_1 \neq 0,$ then $x_0x_1^{w_1}$ is divisible by $m$ and notice that $\frac{m}{x_0x_1^{w_1}}$ $ \in    (x_1^{w_1},x_1x_2^{w_2},\ldots,  x_1x_r^{w_r}, f_1,\\\ldots,f_t)^{s-1} \cap  (x_0,x_2^{w_2},\ldots,x_r^{w_r}, f_1,\ldots,f_t)^{s-1}= I(D)^{(s-1)}.$ Hence by induction hypothesis	$\frac{m}{x_0x_1^{w_1}}$ $ \in  I(D)^{s-1}$ and so $m \in I(D)^s.$

If $b_2 \neq 0,$ then $x_1x_2^{w_2}$ is divisible by $m$ and observe that $\frac{m}{x_1x_2^{w_2}}$ $ \in   (x_1^{w_1},x_1x_2^{w_2},\ldots,  x_1x_r^{w_r}, f_1,\\\ldots,f_t)^{s-1} \cap  (x_0,x_2^{w_2},\ldots,x_r^{w_r}, f_1,\ldots,f_t)^{s-1} = I(D)^{(s-1)}.$ Hence by induction hypothesis	$\frac{m}{x_1x_2^{w_2}}$ $ \in  I(D)^{s-1}$ and so $m \in I(D)^s.$ Similarly if $b_i \neq 0$ for some $i \in \{3,\ldots,r\}$, we can show    $m \in I(D)^s.$ If $b_1=\cdots=b_r=0,$ then $m_2 \in I(D)^s$ and hence $m \in I(D)^s.$	
	
\textbf{Case (2)} Assume that $a_1 = 0.$
Then $m_1 \in I(D)^s$ and so $m \in I(D)^s.$
\end{proof}
\vspace*{1cm}

%

\textbf{Acknowledgements}
\vspace*{0.15cm}\\
We would like to thanks the anonymous referee for their helpful  suggestions. The first author was supported by SERB(DST) grant No.: $\mbox{MTR}/2020/000429$, India.

\end{document}